\pdfoutput=1
\documentclass[bj, preprint]{imsart}
\usepackage[colorlinks,citecolor=blue,urlcolor=blue]{hyperref}
\usepackage[numbers]{natbib}
\usepackage{amssymb,amsmath,amsthm}
\usepackage{xcolor}
\usepackage{enumitem}
\usepackage{bm}
\usepackage{logistic}
\usepackage{graphicx} 
\graphicspath{{prebuiltimages/}}
\usepackage{subcaption}
\usepackage{tikz}
\usetikzlibrary{snakes,arrows,shapes,calc,intersections}
\setattribute{journal}{name}{}
\crefformat{equation}{Eq.~(#2#1#3)} 
\Crefformat{equation}{Equation~(#2#1#3)} 
\Crefformat{claim}{Claim~#2#1#3}

\begin{document}

\begin{frontmatter}
\title{On Lasso refitting strategies}
\runtitle{On Lasso refitting strategies}
\begin{aug}
    \author{\fnms{Evgenii} \snm{Chzhen}\thanksref{a,e1}%
    \ead[label=e1, mark]{evgenii.chzhen@univ-paris-est.fr}},
    \author{\fnms{Mohamed} \snm{Hebiri}\thanksref{a,e2}%
    \ead[label=e2, mark]{mohamed.hebiri@u-pem.fr}},
    \and
    \author{\fnms{Joseph} \snm{Salmon}\thanksref{c, e3}%
    \ead[label=e3, mark]{joseph.salmon@umontpellier.fr}%
    }
    \runauthor{E. Chzhen et al.}
    \affiliation{Universit\'e Paris-Est and T\'el\'ecom ParisTech\\}
    \address[a]{LAMA, Universit\'e Paris-Est,
    5 Boulevard Descartes, 77420, Champs-sur-Marne, France
    \printead{e1,e2}}
    \address[c]{IMAG, Universit\'e de Montpellier, Place Eug\`ene Bataillon,
    34095 Montpellier Cedex 5,
    France\\
    \printead{e3}}
\end{aug}

\begin{abstract}

A well-know drawback of $\ell_1$-penalized estimators is the systematic shrinkage of the large coefficients towards zero.
A simple remedy is to treat Lasso as a model-selection procedure and to perform
a second refitting step on the selected support.
In this work we formalize the notion of refitting and provide oracle bounds for arbitrary refitting procedures of the Lasso solution.
One of the most widely used refitting techniques which is based on Least-Squares may bring a problem of interpretability, since the signs of the refitted estimator might be flipped with respect to the original estimator.
This problem arises from the fact that the Least-Squares refitting 
considers  only the support of the Lasso solution, avoiding any information about signs or amplitudes.
To this end we define a sign consistent refitting as an arbitrary refitting procedure, preserving the signs of the first step Lasso solution and provide Oracle inequalities for such estimators.
Finally, we consider special refitting strategies: \Bregman and \Boosted.
\Bregman has a fruitful property to converge to the Sign-Least-Squares refitting (Least-Squares with sign constraints), which provides with greater interpretability.
We additionally study the \Bregman refitting in the case of orthogonal design, providing with simple intuition behind the proposed method.
\Boosted, in contrast, considers information about magnitudes of the first Lasso step and allows to develop better oracle rates for prediction.
Finally, we conduct an extensive numerical study to show advantages of one approach over others in different synthetic and semi-real scenarios.

\end{abstract}


\end{frontmatter}


\section{Introduction}
\label{sec:introduction}
Least absolute shrinkage and selection operator (Lasso), introduced by~\citet{Tibshirani96}, became a popular method in high dimensional statistics due to competitive numerical solvers (as it is a convex program) and fruitful statistical guaranties~\citep{Bickel_Ritov_Tsybakov09,Koltchinskii11}.
However, the shrinkage of large magnitudes towards zero, observed in practice, may affect the overall conclusion about the model.
Different remedies were proposed to overcome this affect, all of them having their advantages and disadvantages.
For instance, one may consider a non-convex penalty instead of the $\ell_1$ regularization~\citep{Fan_Li01,Zhang10,Gasso_Rakotomamonjy_Canu08}: this approach increases a computational complexity and might be not applicable in large-scale scenarios.
Another way to avoid the underestimation of the coefficients is to perform the second least-squares refitting step based on the first step Lasso solution~\citep{Belloni_Chernozhukov13,Lederer13}: such an approach brings the problem of interpretability, since the coefficients may switch signs with respect to the original Lasso solution.
A lot of theoretical and applied works are devoted to the study of least-squares refitting of an arbitrary first step estimator~\citep{Belloni_Chernozhukov13,Lederer13,Deledalle_Papadakis_Salmon15,Deledalle_Papadakis_Salmon_Vaiter16}.

Unlike such approaches we are rather interested in general refitting strategies of the Lasso estimator.
{ We provide a natural definition of what general refitting is, in the sense that we aim at reducing the data-fitting term of the original Lasso estimator. Moreover, since our initial (first step) estimator is the Lasso, our approach
allows us to use previous theoretical analysis provided for Lasso to derive guarantees for a wide class of refitting (second step) strategies.}
In~\Cref{sec:notation} we introduce notation, used throughout the article, and the basic Lasso theory is partly covered in~\Cref{subsec:lasso_theory}.
Readers who are familiar with the Lasso theory may skip~\Cref{subsec:lasso_theory} and proceed to the following sections.
\Cref{subsec:refitting_strategy}, is concerned with our theoretical framework, where we define a refitting strategy as an estimator which reaches a lower mean square error (MSE), compared to the first step Lasso solution.
For this family of refitting strategies we show that the rates for prediction are bounded by the Lasso rates plus an $\ell_1$-norm of difference between Lasso estimator and the refitted vector.
Inspired by this result we propose to use an additional information, provided by the Lasso solution, for refitting.
It leads to a least-squares refitting with constraints, to avoid an explosion of the refitted coefficients.
This estimator can be seen as \Boosted strategy~(see \Cref{sec:lasso_boosting} and more particularly~\Cref{lem:boosting_lemma}) allows us to develop better prediction bounds compared to the classical Lasso bounds.

Additionally, we propose another family of refitting strategies in~\Cref{subsec:sign_consistent_refitting_strategy}, which restricts the possibility to switch signs with respect to the first step Lasso solution in addition to lower MSE.
For every refitting strategy in this family we provide a unified oracle inequality stated in~\Cref{thm:oracle_sign_refitting_lasso} showing minimax rates under the same assumptions as oracle inequalities for Lasso.
We introduce \Bregman, which can be seen as a generalization of Bregman Iterations~\citep{Osher_Burger_Goldfarb_Xu_Yin05,osher2016sparse}, widely used method in compressed sensing settings and has a strong connection with the method proposed by~\citet{Brinkmann_Burger_Rasch_Satour16}.
Analyzing the \Bregman in case of denoising model (\Cref{sec:linear_regression_with_orthogonal_design}) we provide useful insights on the proposed method.
Additionally, we show that \Bregman is a refitting strategy converging to \SLSLasso, which can be tracked back in~\citep{Brinkmann_Burger_Rasch_Satour16}.
This estimator restricts the possibility to flip signs, minimizing MSE meanwhile.
For \Bregman, we conduct an intensive analysis in the orthogonal design case which is of independent interest and that exhibits some interesting interpretation of this method, and makes some analogies between \Bregman and well-known existing thresholding methods (such as soft/hard/firm-thresholding).

Finally, we conduct an extensive numerical study of different post-Lasso refitting strategies to show advantages of different estimators in various scenarios.
{ Let us conclude this section by summarizing our main contributions, in this paper we aim at:
\begin{itemize}
\item defining formally a refitting Lasso estimator,
\item introducing specific refitting methods that exploit particular properties, such as preserving the Lasso signs, constraining the coefficients amplitudes, etc.
\item providing oracle inequalities for particular refitting strategies, such as Bregman iteration, Sign-Least-Squares refitting and the Boosted refitting,
\item providing further understanding of the Bregman Iterations \cite{Osher_Burger_Goldfarb_Xu_Yin05}.
\end{itemize}}


\section{Framework and notation}
\label{sec:notation}

The standard Euclidean norm is written $\|\cdot\|_2$, the $\ell_1$-norm $\|\cdot\|_1$, and the $\ell_\infty$-norm $\|\cdot\|_\infty$.
For any integer $d \in \bbN$, we denote by $[d]$ the set $\{1, \ldots, d\}$ and by $Q^\top$ the transpose of a matrix $Q$, and $I_d \in \bbR^{d\times d}$ is the identity matrix of size $d$.
For two real numbers $a, b \in \bbR$ we defined by $a \vee b$ the maximum between $a$ and $b$.
For any vectors $a, b \in \bbR^p$ we denote by $\scalar{a}{b} = a^{\top}b$ the Euclidean inner product and by $a \odot b$ the element wise (Hadamard) product of two vectors.
Our approach is valid for a broad class of models, but to avoid digression, we study the prediction performance of the Lasso and refitting strategies only for Gaussian linear regression models with deterministic design. More specifically, we consider $n$ random observations $y_1,\ldots, y_n \in \bbR$ and fixed covariates $x_1,\dots, x_n \in \bbR^p$.
We further assume that there is a regression vector $\tbeta \in \bbR^p$ which satisfies the following relation:
\begin{equation}
    \label{eq:linear_regression_model}
    y = X\tbeta + \varepsilon,\qquad \varepsilon \sim \mathcal{N}(0, \sigma^2 I_n)\enspace,
\end{equation}
where $y = (y_1, \ldots, y_n)^\top \in \bbR^n$ is the response vector and $X = (x_1, \ldots, x_n)^\top \in \bbR^{n \times p}$ the design matrix.
We additionally assume, that the columns of $X$ are normalized in such a way that for all $j \in [p]$ we have $\normin{X_j}_2^2 = n$, where $X_j$ is $j^{th}$ column of the matrix $X$.
For any set $E \subset [p]$,  we denote by $E^c$ the complement to $E$ (\ie $E\cup E^c = [p]$) and by $X_E$ the matrix obtained from the matrix $X$ by erasing all the columns whose indexes are not in $E$.
Similarly, for any $\beta \in \bbR^p$ we write $\beta_E$ to denote the vector obtained from $\beta$ by erasing all the components whose indexes are not in $E$.
For all vectors $\beta \in \bbR^p$ we write $\supp(\beta) \subset [p]$ for the support of the vector $\beta$, \ie $\supp(\beta) = \{j \in [p]: \beta_j \neq 0\}$.
For every real-valued function $f : \bbR^p \mapsto \bbR$ we say that $g \in \bbR^p$ is a subgradient of $f$ at $x \in \bbR^p$ if $f(y) \geq f(x) + \scalar{g}{y - x}$ for all $y \in \bbR^p$.
The set of all subgradients of $f$ at $x \in \bbR^p$ is called subdifferential of $f$ at $x \in \bbR^p$ and written as $\partial f(x)$.
We also remind, that the subdifferential of the $\ell_1$-norm $\partial\norm{\cdot}_1$ is a set valued vector function $\sign(\cdot) = (\sign(\cdot)_1,\ldots,\sign(\cdot)_p)^\top$, defined element-wise by
\begin{equation}
    \label{eq:definition_signs}
    \forall\beta\in\bbR^p,\,\forall j \in [p], \quad \sign(\beta)_j = \begin{cases}
                        \{ 1 \},\quad &\beta_j > 0\\
                        \{ -1 \}, \quad &\beta_j < 0\\
                        [-1, 1],\quad &\beta_j = 0
                     \end{cases}\enspace.
\end{equation}
Also, we assume that the unknown vector $\tbeta$ is sparse, \ie $\supp(\tbeta) = S$ has small cardinality $s$ compared to $n$ and $p$.
To estimate ${\beta}^*$ we first minimize the negative log-likelihood with $\ell_1$ penalty~\citep{Tibshirani96}, which is equivalent for a fixed $\lambda > 0$ to the following optimization problem
\begin{align}
    \label{eq:Lasso}
    \hbeta \in \argmin_{\beta \in \bbR^p} \frac{1}{2n}\norm{y - X\beta}_2^2 + \lambda{\norm{\beta}_1}\enspace.
\end{align}
We also remind the Karush-Kuhn-Tucker (KKT) conditions for \Cref{eq:Lasso}:
\begin{lemma}[KKT conditions for Lasso]
  \label{lem:kkt_lasso}
  The Karush-Kuhn-Tucker conditions~\citep{Boyd_Vandenberghe04} for the Lasso problem~\Cref{eq:Lasso} read as follows:
  $0 \in \lambda\sign(\hbeta) - \frac{1}{n}X^\top(y - X\hbeta),$
  or equivalently: there exists $\hrho \in \sign(\hbeta)$, such that $\label{eq:kkt_lasso_good_form}
        \frac{1}{n}X^\top(y - X\hbeta) = \lambda\hrho.$
\end{lemma}



\section{Lasso theory}
\label{subsec:lasso_theory}

In this section we provide one of the classical Lasso oracle inequalities.
To this end, we introduce the restricted eigenvalue condition~\citep{Bickel_Ritov_Tsybakov09}, a widely used assumption on the design matrix $X$.

\begin{definition}[\citet{Bickel_Ritov_Tsybakov09}]
	We say that $X \in \bbR^{n\times p}$ satisfies the Restricted Eigenvalue condition RE($c_0$, $s$), where $c_0 > 0$ and $s \in [p]$, if $\exists \kappa(c_0, s) > 0$ such that for all $J \subset [p]$ with $|J| \leq s$ we have for all $\Delta \in \bbR^p$
	\begin{equation*}
		\normin{\Delta_{J^c}}_1 \leq c_0 \normin{\Delta_{J}}_1\quad \implies \quad \frac{\normin{X\Delta}_2^2}{n\normin{\Delta_J}_2^2} \geq \kappa^2(c_0, s)\enspace.
	\end{equation*}
\end{definition}
We also state below some classical concentration bound on tail of sup of Gaussian random variables.
\begin{lemma}
	\label{lem:union_bound}
	Let $\varepsilon \sim \mathcal{N}(0, \sigma^2I_n)$ and $X \in \bbR^{n \times p}$ be such that $\forall j \in [p]$ we have $\normin{X_j}_2^2 = n$, hence with probability at least $1 - \delta$ we have
	\begin{align*}
		\norm{X^\top \varepsilon / n}_{\infty} \leq {\lambda} / 2\enspace,
	\end{align*}
	where $\lambda = 2\sigma\sqrt{{2\log{(p/\delta)}}/{n}}$.
\end{lemma}

The following theorem is a starting point of our analysis. Therefore, and for the sake of completeness we state its proof in the Appendix. We mention that similar techniques of the proof can be found in~\citep{Giraud_stat_book14,Dalalyan_Hebiri_Lederer17}.

\begin{theorem}
	\label{thm:oracle_lasso}
	If the design matrix $X \in \bbR^{n \times p}$ satisfies the restricted eigenvalue condition RE($3$, $s$) and ${\lambda} = 2\sigma\sqrt{{2\log{(p/\delta)}}/{n}}$ for every $\delta \in (0, 1)$, then with probability $1 - \delta$ the following bound holds
	\begin{equation*}
			\frac{1}{n}\normin{X(\beta^* - \hbeta)}_2^2 \leq \frac{9\lambda^2 s}{4 \kappa^2(3, s)}\enspace,
	\end{equation*}
	where $\hbeta$ is a Lasso solution with tuning parameter $\lambda$.
\end{theorem}
\begin{proof} See supplementary materials for the proof.
\end{proof}

\section{Refitting strategies}
\label{subsec:refitting_strategy}
Underestimation of large coefficients by Lasso and other $\ell_1$-penalized estimators have long been well known by practitioners, and simple remedies
have been proposed on case by case analysis.
One of such approaches is Least-Squares refitting - widely used in high-dimensional regression to reduce the bias of the coefficients and consists in performing a least-squares re-estimation of the non-zero coefficients of the solution.
Such a procedure is theoretically analyzed by~\citet{Belloni_Chernozhukov13} applied to an arbitrary first-step estimator.
\Citet{Lederer13}, showed that blind Least-Squares refitting of the Lasso solution is not advised in all possible scenarios and developed a refitting criteria.
\Citet{Deledalle_Papadakis_Salmon15,Deledalle_Papadakis_Salmon_Vaiter16} are mostly concerned with the practical aspects of refitting and provide efficient numerical procedures to perform refitting simultaneously along with the original estimator.
In contrast, here we are interested in arbitrary refitting of the Lasso solution, which allows to exploit Lasso theory to provide new insights.
In this section we study a general framework for refitting techniques.
We define a refitting of a Lasso solution $\hbeta$ as:

\begin{definition}
	\label{def:refititng_vector}
	Let $\hbeta = \hbeta(\lambda)$ be a Lasso solution with regularization $\lambda$. We call a vector $\bbeta = \bbeta(\hat\beta, X, y)$ a refitting of the Lasso $\hbeta$ if it reduces the original loss function, namely if:
	\begin{align}
	    \normin{y - X\bbeta}_2 &\leq \normin{y - X\hbeta}_2\enspace.
	\end{align}
\end{definition}
{Several known refitting strategies are falling inside of this framework, for instance, Least-Squares Lasso and Relaxed Lasso~\citep{MEINSHAUSEN07}.
Relaxed Lasso was introduced by~\citet{MEINSHAUSEN07} and is defined for any positive $\lambda$ and $\phi \in [0, 1]$ as a solution to the following convex problem:
\begin{align}
    \label{eq:relaxed_lasso}
    \bbeta^{\lambda, \phi} \in \argmin_{\beta : \supp(\beta) \subset \supp(\hbeta^{\lambda})} \frac{1}{2n}\norm{y - X\beta}_2^2 + \phi\lambda{\norm{\beta}_1}\enspace,
\end{align}
where $\hbeta^{\lambda}$ is the first step Lasso solution with parameter $\lambda$.
To see that the relaxed Lasso in~\cref{eq:relaxed_lasso} is a refitting strategy in the sense of~\Cref{def:refititng_vector} it is sufficient to notice that:
\begin{align*}
    \frac{1}{2n}\norm{y - X\bbeta^{\lambda, \phi}}_2^2 + \phi\lambda{\norm{\bbeta^{\lambda, \phi}}_1} &\leq \frac{1}{2n}\norm{y - X\hbeta^{\lambda}}_2^2 + \phi\lambda{\norm{\hbeta^{\lambda}}_1}\enspace;\\
    \frac{1}{2n}\norm{y - X\hbeta^{\lambda}}_2^2 + \lambda{\norm{\hbeta^{\lambda}}_1} &\leq  \frac{1}{2n}\norm{y - X\bbeta^{\lambda, \phi}}_2^2 + \lambda{\norm{\bbeta^{\lambda, \phi}}_1}\enspace.
\end{align*}
Summing these inequalities and using the fact that $\lambda - \lambda\phi \ge 0$ we obtain the required condition in~\Cref{def:refititng_vector}.}

We emphasize, that one should not expect a superior performance of a refitting from ~\Cref{def:refititng_vector}, since the only information available for $\bbeta$ is the mean square error.
Additionally notice, that the least-squares solution is obviously a refitting strategy for every Lasso solution, which justifies the previous remark.
However, defining and analyzing such a family provides with interesting insights and serves as the step towards more thoughtful refitting strategies which are discussed in~\Cref{sec:lasso_boosting}.
\begin{theorem}
	\label{thm:oracle_refitting_lasso}
	For ${\lambda} = 2\sigma\sqrt{{2\log{(p/\delta)}}/{n}}$, where $\delta \in (0, 1)$ is arbitrary, with probability $1 - \delta$ the following bound holds
	\begin{equation*}
			\frac{1}{n}\normin{X(\beta^* - \bbeta)}_2^2 \leq \frac{1}{n}\normin{X(\tbeta - \hbeta)}_2^2 + {\lambda}\normin{\bbeta - \hbeta}_1\enspace,
	\end{equation*}
	where $\bbeta$ is a refitting of Lasso solution $\hbeta$.
\end{theorem}

Previous theorem shows, that by controlling the $\ell_1$-distance between refitting and Lasso solution, one might obtain satisfying performance, we discuss this idea in~\Cref{sec:lasso_boosting}.
{ In what follows, we consider specific refitting methods in order to refine prediction error bounds.
First, we consider the sign consistent refitting family that shares the same sign vector as the initial Lasso estimator.
Then, we develop the Boosted Lasso, a refitting method that aims at reducing the $\ell_1$ distance with the initial Lasso estimator.
Finally, we introduce the Bregman Lasso, which somehow generalizes the Bregman Iterations \cite{Osher_Burger_Goldfarb_Xu_Yin05} by adding some flexibility on the tuning parameters.
We also highlight its connections with the Sign-Least-Squares Lasso, a particular sign consistent refitting strategy.}

\subsection{Sign consistent refitting strategies}
\label{subsec:sign_consistent_refitting_strategy}

Previous section is concerned with an arbitrary refitting strategy, which only uses the information about the mean square error of the Lasso solution.
Here, we are interested in a more sophisticated family of refitting strategies, which additionally exploits the information provided by the sign of the Lasso solution.
Such an approach has some similarities with the methods introduced by~\Citet{Brinkmann_Burger_Rasch_Satour16}, even thought the authors had a different motivation.

\begin{definition}
	\label{def:sign_refitting_vector}
	Let $\hbeta = \hbeta(\lambda)$ be a Lasso solution with regularization $\lambda$. We call a $\bbeta = \bbeta(\hat\beta, X, y)$ sign-consistent refitting of the Lasso solution $\hbeta$ if
	\begin{align}
			\label{eq:refitting}
			\normin{y - X\bbeta}_2 \leq \normin{y - X\hbeta}_2\quad&\text{(refitting)}\enspace,\\
			\label{eq:sign_refitting}
			\frac{1}{n}X^\top(y - X\hbeta) \in \lambda\sign(\bbeta)\quad&\text{(sign-consistency)}\enspace.
	\end{align}
\end{definition}
\begin{remark}
	\label{rem:sign_refitting}
	The sign consistency property in~\Cref{def:sign_refitting_vector} ensures that a sign-consistent refitting vector $\bbeta$ satisfies for all $j \in [p]$ the following conditions
	\begin{itemize}
		\item if $n^{-1}X^\top_j(y - X\hbeta) = \lambda$, hence $\bbeta_j \geq 0$;
		\item if $n^{-1}X^\top_j(y - X\hbeta) = -\lambda$, hence $\bbeta_j \leq 0$;
		\item if $n^{-1} |X^\top_j(y - X\hbeta)| < \lambda$, hence $\bbeta_j = 0$;
	\end{itemize}
\end{remark}
Notice, that the equations in the previous remark are exactly the first-order optimality conditions for the Lasso problem written component-wise.
The next theorem shows that the definition of the sign-refitting allows to develop oracle rates without any additional assumptions, except classical ones used for Lasso bounds.

\begin{theorem}
	\label{thm:oracle_sign_refitting_lasso}
	If the design matrix $X \in \bbR^{n \times p}$ satisfies the restricted eigenvalue condition RE($3$, $s$) and ${\lambda} = 6\sigma\sqrt{{2\log{(p/\delta)}}/{n}}$ for some $\delta \in (0 , 1)$, then with probability $1 - \delta$ the following bound holds
	\begin{align*}
	    \frac{1}{n}\normin{X(\tbeta - \bbeta)}_2^2 + \frac{1}{n}\normin{X(\tbeta - \hbeta)}_2^2 \leq \lambda^2s\Big(\frac{9}{\kappa^2(3, s)} + \frac{16}{\kappa^2(1, s)}\Big)\enspace,
	\end{align*}
	where $\bbeta$ is a sign-consistent refitting of Lasso solution $\hbeta$.
\end{theorem}

The proof of this theorem is based on~\Cref{def:sign_refitting_vector}, to be more precise, we can use an additional equality due to the sign-consistency property in~\cref{eq:sign_refitting}, that is:
\begin{align*}
	\frac{1}{n}(\bbeta - \hbeta)^\top X^\top(y - X\hbeta) &= \lambda\scalar{\bbeta - \hbeta}{\hrho} = \lambda(\normin{\bbeta}_1 - \normin{\hbeta}_1)\enspace.
\end{align*}
This relation holds, since two estimators $\hbeta$ (Lasso) and $\bbeta$ (sign-refitting) share the same subgradient.
Hence, we are able to re-use proof techniques similar to the one of~\Cref{thm:oracle_lasso} to provide an oracle inequality.


\subsection{\Boosted}
\label{sec:lasso_boosting}
Notice that to prove~\Cref{thm:oracle_refitting_lasso}, we used only the refitting property~\Cref{def:refititng_vector}, one of the possible candidates is the following refitting step
\begin{align}
    \label{eq:estimator_without_the_name}
     \bbeta \in \argmin_{\beta \in \Gamma} \frac{1}{2n}\normin{y - X\beta}_2^2 \enspace,
\end{align}
where $\Gamma = \{\beta \in \bbR^p : \normin{\beta - \hbeta}_1 \leq \hat{s}\lambda\}$
and $\hat s = |\supp(\hbeta)|$.
Intuitively, the Lasso coefficients are shrunk towards zero by a value proportional to the tuning parameter $\lambda$.
Since there are $\hat s$ non-zero coefficient in the Lasso solution, the proposed refitting strategy tries to "unshrink" $\hat s$ non-zero coefficients.
As we measure the "shrinkage" factor globally with the $\ell_1$-norm, it is natural (inspired by the orthogonal design) to set this factor as $\hat s \lambda$.
This motivates our choice of the feasible set $\Gamma$.
This estimator is legitimate following results given in~\citep[Theorem 7.2]{Bickel_Ritov_Tsybakov09} :
\begin{theorem}[\citet{Bickel_Ritov_Tsybakov09}]
    \label{thm:hat_sparsity_of_lasso}
    If the design matrix $X \in \bbR^{n \times p}$ satisfies the restricted eigenvalue condition RE($3$, $s$) and ${\lambda} = 2\sigma\sqrt{{2\log{(p/\delta)}}/{n}}$ for some $\delta \in (0, 1)$, then with probability $1 - \delta$ the following bound holds
    \begin{align}
        \hat s \leq \frac{9s\phi_{\max}}{\kappa^2(3, s)}\enspace,
    \end{align}
    where $\phi_{\max}$ is the maximal eigenvalue of $X^\top X/n$.
\end{theorem}
The proof of this theorem is a direct application of Theorem~7.2 in~\citep{Bickel_Ritov_Tsybakov09} together with ~\Cref{thm:oracle_lasso} to bound $\frac{1}{n}\normin{X(\beta^* - \hat\beta)}_2^2$, we then omit it here.
For the estimator, described in~\Cref{eq:estimator_without_the_name}, we can state the following result, which relies on using both~\Cref{thm:hat_sparsity_of_lasso} and~\Cref{thm:oracle_refitting_lasso}:
\begin{cor}
    \label{cor:oracle_for_constrained_boosted}
    If the design matrix $X \in \bbR^{n \times p}$ satisfies the restricted eigenvalue condition RE($3$, $s$) and ${\lambda} = 2\sigma\sqrt{{2\log{(p/\delta)}}/{n}}$ for some $\delta \in (0, 1)$, then with probability $1 - \delta$ the following bound holds
    \begin{equation*}
            \frac{1}{n}\normin{X(\beta^* - \bbeta)}_2^2 \leq \frac{\lambda^2s}{\kappa^2(3, s)}\Big(\frac{9}{4}+9\phi_{\max}\Big)\enspace.
    \end{equation*}
\end{cor}

Note that the control over the magnitudes of the coefficients allows to develop desirable rates for the refitted estimator.
The description of the set $\Gamma$ in the definition of the refitting~\eqref{eq:estimator_without_the_name} motivates us to consider the following \Boosted estimator \citep{Buhlmann_Yu03}:
\begin{definition}[\Boosted]
For any $\lambda_1, \lambda_2 > 0$ we call $\hbeta^{\lambda_1, \lambda_2}$ a \Boosted refitting if it is a solution of
    \begin{align}
    \label{eq:lasso_boosting_strange_form}
    \hbeta^{\lambda_1, \lambda_2} \in \argmin_{\beta \in \bbR^p} \frac{1}{2n}\norm{y - X\beta}_2^2 + \lambda_2\normin{\beta - \hbeta^{\lambda_1}}_1\enspace,
\end{align}
where $\hbeta^{\lambda_1}$ is the Lasso solution with tuning parameter $\lambda_1$.
\end{definition}
\begin{remark}
    {Indeed, procedure~\eqref{eq:lasso_boosting_strange_form} consists in applying the \emph{Twicing} method \cite{Tukey77} to the Lasso, which is is equivalent to a two-step boosting approach.} To see that, let $\hat\Delta = \beta - \hbeta$. Then, with a change of variable we get:
    \begin{align}
        \label{eq:lasso_boosting}
        \hat\Delta \in \argmin_{\Delta \in \bbR^p} \frac{1}{2n}\norm{(y - X\hbeta^{\lambda_1}) - X\Delta}_2^2 + \lambda_2\normin{\Delta}_1\enspace,
    \end{align}
    and finally $\hbeta^{\lambda_1, \lambda_2} = \hat\Delta + \hbeta^{\lambda_1}$.
\end{remark}
It is known result that in the Lasso case there exists a critical value $\lambda_{1, \max} = \normin{X^\top y / n}_{\infty}$ such that $\hbeta^{\lambda_1} = 0$ iff $\lambda_1 \geq \lambda_{1, \max}$, due to the previous remark, the \Boosted can be written as a Lasso problem and we can give the result of the same nature:
\begin{proposition}
\label{prop:GoodProp}
    If $\lambda_1 < \lambda_{1, \max}$, then the solution of \Boosted satisfies: $\hbeta^{\lambda_1, \lambda_2} = \hbeta^{\lambda_1}$ iff $\lambda_2 \geq \lambda_1$. Moreover, if $\lambda_1 \geq \lambda_{1, \max}$, then the \Boosted estimator is simply the Lasso estimator with tuning parameter $\lambda_2$.
\end{proposition}
    Control over the $\ell_1$-norm of the difference, allows to develop oracle inequality with minimax rates (\Cref{cor:oracle_for_constrained_boosted}), similar result but somehow stronger can be shown for the regularized version (\Boosted):
\begin{lemma}
    \label{lem:boosting_lemma}
    For any $\hbeta^{\lambda_1}$ on the event $\{\normin{\varepsilon^\top X/n}_{\infty} \leq {\lambda_1}/{2}\}$ we have:
        \begin{align}\label{ineq:boosting_lemma}
        \frac{1}{n}\norm{X(\tbeta - \hbeta^{\lambda_1, \lambda_2})}_2^2 + (2\lambda_2 - \lambda_1)\normin{\hbeta^{\lambda_1, \lambda_2} - \hbeta^{\lambda_1}}_1\leq \frac{1}{n}\norm{X(\tbeta - \hbeta^{\lambda_1})}_2^2\enspace.
    \end{align}
\end{lemma}

Combining~\Cref{lem:boosting_lemma} and~\Cref{thm:oracle_lasso} we can state the following corollary
\begin{cor}
    \label{cor:boosted_oracle}
    If the design matrix $X \in \bbR^{n \times p}$ satisfies the restricted eigenvalue condition RE($3$, $s$) and ${\lambda_1} = 2\sigma\sqrt{{2\log{(p/\delta)}}/{n}}$ for some $\delta \in (0, 1)$, then with probability $1 - \delta$ the following bound holds
    \begin{equation*}
            \frac{1}{n}\normin{X(\beta^* - \bbeta)}_2^2 + ({2\lambda_2 - \lambda_1})\normin{\bbeta - \hbeta}_1 \leq \frac{9\lambda^2_1 s}{4 \kappa^2(3, s)}\enspace,
    \end{equation*}
    where $\bbeta = \hbeta^{\lambda_1, \lambda_2}$ is a boosting refitting of Lasso solution $\hbeta = \hbeta^{\lambda_1}$.
\end{cor}

The \Boosted is obviously a refitting strategy, and the previous result shows that it is worth refitting in terms of prediction error.
Moreover,~\Cref{prop:GoodProp} and~\Cref{cor:boosted_oracle} suggest to select $\lambda_2 \in (\lambda_1/2, \lambda_1)$, such a choice allows to improve the Lasso prediction accuracy with high probability.
Besides, the result can be applied iteratively as it does not depend on the choice of $\hbeta^{\lambda_1}$, hence it works for any iteration step.
\begin{remark}
    {A possible extension of the \Boosted is the Boosted Support Lasso:
        \begin{align}
            \label{eq:boosted_support_lasso}
            \bbeta^{\lambda_1, \lambda_2} \in \argmin_{\beta : \supp(\beta) \subset \supp(\hbeta^{\lambda_1})} \frac{1}{2n}\norm{y - X\beta}_2^2 + \lambda_2{\normin{\beta - \hbeta^{\lambda_1}}_1}\enspace,
        \end{align}
        this estimator is inspired by the Relaxed Lasso in~\Cref{eq:relaxed_lasso} and the \Boosted in~\Cref{eq:lasso_boosting_strange_form}.
        Interestingly, both~\Cref{lem:boosting_lemma} and~\Cref{cor:boosted_oracle} hold for the Boosted Support Lasso refitting and follow completely identical proof.}
\end{remark}


\subsection{\Bregman}
\label{sec:bregman_iteration}
We first remind the definition of the Bregman divergence associated with the $\ell_1$-norm
\begin{definition}[Bregman divergence for the $\ell_1$-norm]
  For any $z, w \in \bbR^p$ and any $\rho \in \partial\norm{w}_1$, the Bregman divergence for the $\ell_1$-norm is defined as
\begin{equation}
    \label{eq:definition_bregman_divergence}
    \Breg{\ell_1}^{\rho}(z, w) = \norm{z}_1 - \normin{w}_1 - \scalar{\rho}{z - w} \geq 0, \quad \rho \in \partial\norm{w}_1\enspace.
\end{equation}
\end{definition}
In ~\citep{Osher_Burger_Goldfarb_Xu_Yin05}, the authors proposed the Bregman Iterations procedure, originally designed to improve iso-TV results.
In Lasso case, Bregman Iterations has the following expression: for a fixed $\lambda > 0$, initializing with $\hrho_0= \hbeta_0  = 0_p$,
\begin{equation}
\label{eq:old_bregman_iteration}
\begin{aligned}
    \hbeta_k &\in \argmin_{\beta \in \bbR^p} \frac{1}{2n}\norm{y - X\beta}_2^2 + \lambda\Breg{\ell_1}^{\hrho_{k-1}}(\beta, \hbeta_{k-1})\enspace,\\
    \text{s.t.} \quad \hrho_{k} &= \hrho_{k-1} + \frac{1}{\lambda n}X^\top(y - X{\hbeta_{k}})\enspace,
\end{aligned}
\end{equation}
where $\Breg{\ell_1}^{\rho}(\cdot, \cdot)$ is the Bregman divergence defined in~\cref{eq:definition_bregman_divergence}.
The Bregman Iterations can be seen as a discretization of Bregman Inverse Scale Space (ISS), which is analyzed by~\citet{osher2016sparse}, who
 provided statistical guarantees for the ISS dynamic.
One of the drawbacks of such an approach is the iterative nature of the algorithm: one need to tune the number of iterations $k$ and the regularization parameter $\lambda > 0$.
In the recent work by~\citet{Brinkmann_Burger_Rasch_Satour16}, the authors proposed another closely related algorithm, which for a given Lasso solution $\hbeta^{\lambda_1}$, performs the following refitting
\begin{equation}
    \label{eq:bregman_constraint_equal_zero}
    \begin{aligned}
        & {\bbeta}\in \argmin_{{\beta} \in \bbR^p}
          \frac{1}{2n}\norm{y - X\beta}_2^2 \\
        & \text{s.t.}
         \quad \Breg{\ell_1}^{\hrho^{\lambda_1}}(\beta, \hbeta^{\lambda_1}) = 0\enspace.
    \end{aligned}
\end{equation}
Unlike the previous approaches, in this section we consider the Lasso solution $\hbeta^{\lambda_1}$ and the following \Bregman refitting strategy, defined as
\begin{definition}[\Bregman]
For any $\lambda_1, \lambda_2 > 0$ we call $\hbeta^{\lambda_1, \lambda_2}$ a \Bregman refitting a solution of
\begin{align}
    \label{eq:bregman_iteration_linear}
    \hbeta^{\lambda_1, \lambda_2} \in \argmin_{\beta \in \bbR^p} \frac{1}{2n}\norm{y - X\beta}_2^2 + \lambda_2\Breg{\ell_1}^{\hrho^{\lambda_1}}(\beta, \hbeta^{\lambda_1})\enspace,
\end{align}
where $\hbeta^{\lambda_1}$ is a Lasso solution with tuning parameter $\lambda_1$ and $\Breg{\ell_1}^{\hrho^{\lambda_1}}(\cdot, \cdot)$ is the Bregman divergence given in~\Cref{eq:definition_bregman_divergence}.
\end{definition}
Considering the particular case of orthogonal design, we show in~\Cref{sec:linear_regression_with_orthogonal_design} that \Bregman refitting is a generalization of both approaches.
In particular for this design, the \Bregman refitting can recover Bregman Iterations for any $k$.
Also, \Bregman is computationally more appealing since it only requires evaluating two Lasso problems, while the Bregman Iterations would require $k$ evaluations.

We start by introducing some basic properties of the Bregman divergence associated with  the $\ell_1$-norm.
\begin{lemma}
\label{lemma:props_bregman_divergence}
Let ${z}, w \in \bbR^p$, and denote $\rho \in \partial\norm{{w}}_1$ a subgradient of the $\ell_1$-norm evaluated at ${w} \in \bbR^p$, therefore the following properties hold independent of the choice of $\rho$
\begin{enumerate}
    \item ${z} \mapsto \Breg{\ell_1}^{\rho}({z}, w)$ is convex for all $w$,
    \item $\Breg{\ell_1}^{\rho}({z}, w) = \norm{{z}}_1 - \scalar{{\rho}}{{z}}$, \label{prop:breg_via_subgrad}
    \item $0 \leq \Breg{\ell_1}^{\rho}({z}, w) \leq 2\norm{{z}}_1$, \label{prop:breg_bounds}
    \item $\Breg{\ell_1}^{\rho}({z}, w) = \sum_{i = 1}^p |z_i| - z_i\rho_i  = \sum_{i = 1}^p \Breg{\ell_1}^{\rho_i}(z_i, w_i)$, where $z_i, w_i, \rho_i$ are the $i^{th}$ components of ${z}, w, {\rho}$, \label{prop:separ}
    \item If $\sign({z}) = \sign(w)$, therefore $\Breg{\ell_1}^{\rho}({z},w) = 0$.
\end{enumerate}
\end{lemma}
\begin{proof}
    See supplementary material for details.
\end{proof}
\begin{figure}[t!]
\centering
\begin{subfigure}[b]{0.45\textwidth}
\centering
\begin{tikzpicture}[
    scale=0.99,
    thick,
    >=stealth',
    dot/.style = {\
      draw,
      fill = white,
      circle,
      inner sep = 0pt,
      minimum size = 4pt
    }
  ]
  \coordinate (O) at (0,0);
  \clip(-3, -2.1) rectangle (3, 3);
  \draw[->] (-3,0) -- (3,0) coordinate[label = {below: }] (xmax);
  \draw[->] (0,-2) -- (0,3) coordinate[label = {right: }] (ymax);
  \path[name path=x] (0.3,0.5) -- (6.7,4.7);
  \path[name path=y] plot[smooth] coordinates {(-0.3,2) (2,1.5) (4,2.8) (6,5)};
  \scope[name intersections = {of = x and y, name = i}]
    \draw[color=black]      (0,0) -- (2.5,2.5) node[pos=0.8, below right] {};
    \draw[color=black]      (0,0) -- (-2.5,2.5) node[pos=0.8, below right] {};
    \draw[color=blue, dashed]      (2,0) -- (2,2) node[pos=0.35, above right] {$|w|$};
    \draw[dashed] (-1.3,-1.3)--(2,2) node[pos=0, below]{$\rho$};
    \draw [blue,fill=blue] (2, 0) circle(0.35ex);
    \node [blue] at (2, -0.3) {$w$};
    \draw [blue,fill=blue] (2, 2) circle(0.35ex);
    \draw[color=black, dashed]      (-1,-1) -- (-1,1) node[pos=0.5, below left] {};
    \draw [blue,fill=blue] (-1, 0) circle(0.35ex);
    \draw [black,fill=black] (-1, 1) circle(0.35ex);
    \draw [black,fill=black] (-1, -1) circle(0.35ex);
    \node [blue] at (-0.8, 0.2) {$z$};
    \draw[decorate, decoration={brace, mirror, amplitude=5pt, raise=3pt},] (-1,1)-- node[black,midway, below left=0.25cm]
         {$\Breg{\ell_1}^{\rho}(z, w)$}(-1,-1);
   \endscope
\end{tikzpicture}%
\caption{Case $w \neq 0$ and $\sign(z) \neq \sign(w)$}
\label{fig:subgradient_1}
\end{subfigure}%
\begin{subfigure}[b]{0.45\textwidth}
\centering
\begin{tikzpicture}[
    scale=0.99,
    thick,
    >=stealth',
    dot/.style = {
      draw,
      fill = white,
      circle,
      inner sep = 0pt,
      minimum size = 4pt
    }
  ]
  \coordinate (O) at (0,0);
  \clip(-3, -2.1) rectangle (3, 3);
  \draw[->] (-3,0) -- (3,0) coordinate[label = {below: }] (xmax);
  \draw[->] (0,-2) -- (0,3) coordinate[label = {right: }] (ymax);
  \path[name path=x] (0.3,0.5) -- (6.7,4.7);
  \path[name path=y] plot[smooth] coordinates {(-0.3,2) (2,1.5) (4,2.8) (6,5)};
  \scope[name intersections = {of = x and y, name = i}]
    \draw[color=black]      (0,0) -- (2.5,2.5) node[pos=0.8, below right] {};
    \draw[color=black]      (0,0) -- (-2.5,2.5) node[pos=0.8, below right] {};
    \draw[dashed] (-1.5,-0.75)--(2.5,1.25) node[pos=0, below]{$\rho$};
    \draw [blue,fill=blue] (0, 0) circle(0.35ex);
    \node [blue] at (0.3, -0.3) {$w$};
    \draw[color=black, dashed]      (-1,-0.5) -- (-1,1) node[pos=0.5, below left] {};
    \draw [blue,fill=blue] (-1, 0) circle(0.35ex);
    \draw [black,fill=black] (-1, 1) circle(0.35ex);
    \draw [black,fill=black] (-1, -0.5) circle(0.35ex);
    \node [blue] at (-0.8, 0.2) {$z$};
    \draw[decorate, decoration={brace, mirror, amplitude=5pt, raise=3pt},] (-1,1)-- node[black,midway, below left=0.25cm]
         {$\Breg{\ell_1}^{\rho}(z, w)$}(-1,-0.5);
    \endscope
\end{tikzpicture}%
\caption{Case $w = 0$ and some $\rho \in (-1, 1)$}
\label{fig:subgradient_2}
\end{subfigure}
\caption{Geometrical interpretation of Bregman divergence. {On the left plot the subgradient is uniquely defined and the Bregman divergence equals $2|z|$}. On the right plot the subgradient is a set, therefore the Bregman divergence can be any number from $0$ to $2|z|$ depending on the choice of $\rho$.} \label{fig:M1}
\end{figure}
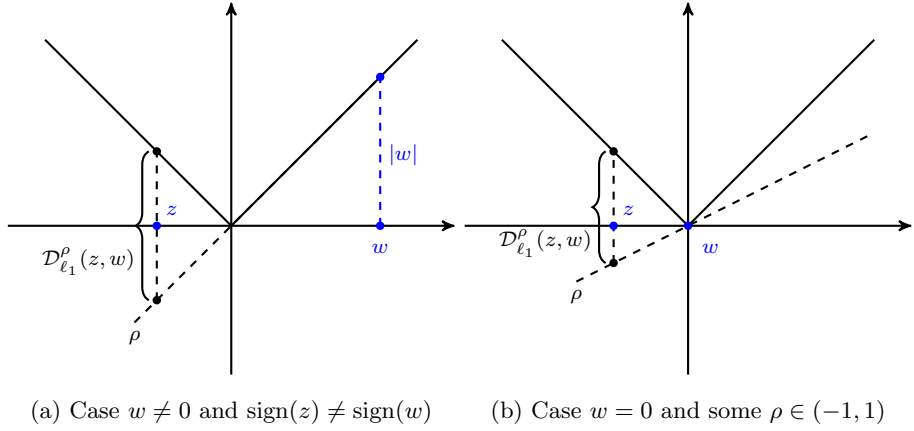
There is a simple geometrical interpretation of Bregman divergence, associated with the $\ell_1$-norm, which is illustrated by \Cref{fig:M1}.
According to Property~\ref{prop:breg_via_subgrad} of ~\Cref{lemma:props_bregman_divergence}, Bregman divergence in~\Cref{eq:bregman_iteration_linear} can be evaluated as
\begin{equation*}
    \Breg{\ell_1}^{\hrho^{\lambda_1}}(\beta, {\hbeta}^{\lambda_1}) = \normin{\beta}_1 - \scalar{\hrho^{\lambda_1}}{\beta}\enspace,
\end{equation*}
where $\hrho^{\lambda_1}$ is a fixed subgradient of the $\ell_1$-norm evaluated at $\hbeta^{\lambda_1}$.
Since, the subdifferential of the $\ell_1$-norm is not uniquely defined when evaluated at zero, it is important to fix the way to pick a subgradient.
Possible, and probably the most obvious way to evaluate the subgradient is to write the KKT conditions for Problem~\eqref{eq:bregman_iteration_linear} and fix ${\hrho}^{\lambda_1}$ as follows
\begin{equation}
    \label{eq:refitting_subgradient}
    {\hrho}^{\lambda_1} = \frac{1}{\lambda_1 n}X^\top(y - X{\hbeta}^{\lambda_1})\enspace.
\end{equation}
Starting from here we stick to this choice of the subgradient.
\begin{proposition}
  \label{prop:bregman_as_Lasso}
    With the choice of the subgradient as in~\Cref{eq:refitting_subgradient} and setting $\bar y = y + \frac{\lambda_2}{\lambda_1}(y - X\hbeta^{\lambda_1})$, the \Bregman refitting step can be evaluated as
    \begin{align}
        \label{eq:bregman_as_Lasso}
        \hbeta^{\lambda_1, \lambda_2} \in \argmin_{\beta \in \bbR^p} \frac{1}{2n}\norm{\bar y - X\beta}_2^2 + \lambda_2\normin{\beta}_1\enspace.
    \end{align}
\end{proposition}
{ \Cref{prop:bregman_as_Lasso} suggests that one can easily compute the Bregman Lasso refitting based on some Lasso solver using a modified response vector. Moreover, it
is important to notice that the Bregman refitting is a particular instance of our proposed framework.}
In particular the following result holds:
\begin{proposition}
    \label{lem:bregman_is_refitting}
    \Bregman refitting
     is a refitting strategy in the sense of~\Cref{def:refititng_vector}.
\end{proposition}
{
\begin{theorem}
  \label{thm:bregman_racle}
  If the design matrix $X \in \bbR^{n \times p}$ satisfies the restricted eigenvalue condition RE($3$, $s$) and ${\lambda_1} = \tfrac{1+2r}{r}2\sigma\sqrt{{2\log{(p/\delta)}}/{n}}$ for some $\delta \in (0, 1)$, then with probability $1 - \delta$ the following bound holds:
  \[
  \frac{r + 2}{r}\frac{1}{n}\normin{X(\tbeta - \bbeta)}_2^2 + \frac{1}{n}\normin{X(\tbeta - \hbeta)}_2^2
      \leq \frac{\lambda^2_1 s}{\kappa^2(1, s)} + \frac{9\lambda^2_1 s}{4\kappa^2(3, s)}\enspace,
    \]
    where $r = \lambda_2/\lambda_1$ and $\bbeta = \hbeta^{\lambda_1, \lambda_2}$ is a Bregman refitting of Lasso solution $\hbeta = \hbeta^{\lambda_1}$.
\end{theorem}
}

{The proof is postponed to \Cref{sec:appendix_proofs}, we only mention that an important step in the proof is to obtain the following inequality:
\begin{align}
    \label{eq:refitting_bregman_quantification}
    \normin{X(\hbeta - \bbeta)}_2^2 + \normin{y - X\bbeta}_2^2 \leq \normin{y - X\hbeta}_2^2\enspace,
\end{align}
which quantifies the improvement for the data-fitting term.
Several interesting consequences can be mentioned based on~\Cref{thm:bregman_racle}:
first notice that the oracle inequality depends weakly on the second parameter $\lambda_2$, which indicates that the Bregman refitting is not sensitive to the second parameter;
also, the rate on the right hand side $9\lambda^2 s/4\kappa^2(3, s)$ is matching the Lasso upper bound up to a constant term (to the cost of a lower probability).
\begin{remark}
  A careful analysis of the proof suggests a stronger result:
  under the assumptions of~\Cref{thm:bregman_racle}, with probability $1 - \delta$ at least one of the following inequalities holds:
  \begin{align*}
    \frac{r + 2}{r}\frac{1}{n}\normin{X(\tbeta - \bbeta)}_2^2 + \frac{1}{n}\norm{X(\tbeta - \hbeta)}_2^2 &\leq \frac{9\lambda^2_1 s}{4\kappa^2(3, s)}\enspace,\\
    \frac{1}{n}\normin{X(\tbeta - \hbeta)}_2^2 &\leq \frac{9\lambda^2_1 s}{4\kappa^2(1, s)}\enspace,
  \end{align*}
    where $r = \lambda_2/\lambda_1$ and $\bbeta = \hbeta^{\lambda_1, \lambda_2}$ is the Bregman refitting of the Lasso solution $\hbeta = \hbeta^{\lambda_1}$.
    Note that although one cannot tell apart which one of these two inequalities holds, the prediction bounds are improved in both cases.
    In case the first inequality is true and the classical Lasso bound is tight the Bregman refitting improves on the Lasso.
    In case the second inequality holds then the Lasso bound from \Cref{thm:oracle_lasso} can be improved using a better constant $\kappa^2(1, s)$ instead of $\kappa^2(3, s)$, which can differ significantly.
\end{remark}
}

\subsubsection{Geometrical interpretation}
\label{subsec:geometrical_interpretation}
It is known that~\eqref{eq:Lasso} can be equivalently written in the following form
\begin{equation}
    \label{eq:MLE_with_ell_1_constrained}
    \begin{aligned}
        & {\hbeta}^{T_1}\in \argmin_{{\beta} \in \bbR^p}
        \frac{1}{2n}\norm{y - X\beta}_2^2 \\
        & \text{s.t.}
        \quad \norm{\beta}_{1} \leq T_1\enspace,
    \end{aligned}
\end{equation}
for some $T_1 = T_1(\lambda_1) \geq 0$.
Similarly, we can write the \Bregman refitting step~\Cref{eq:bregman_iteration_linear} in the constrained form as
\begin{equation}
    \label{eq:refitting_with_ell_1_constrained}
    \begin{aligned}
        & {\hbeta}^{T_1, T_2}\in \argmin_{{\beta} \in \bbR^p}
         \frac{1}{2n}\norm{y - X\beta}_2^2 \\
        & \text{s.t.}
        \quad \Breg{\ell_1}^{\hrho^{T_1}}(\beta, \hbeta^{T_1}) \leq T_2\enspace,
    \end{aligned}
\end{equation}
where $\hbeta^{T_1}$ is a solution of the constrained version given in~\Cref{eq:MLE_with_ell_1_constrained}.
The choice $T_2 = 0$ is considered in~\citep{Brinkmann_Burger_Rasch_Satour16}, which corresponds to $\lambda_2$ being large enough in~\Cref{eq:bregman_iteration_linear}.
To provide a geometrical intuition we consider the simple case where $n = p = 2$.
We denote by $\hrho^{T_1}$ the subgradient of the $\ell_1$-norm evaluated at the first-step estimator ${\hbeta}^{T_1}$.
Assume, that the first-step estimator is given, therefore we consider two principal scenarios (case with negative values could be obtained symmetrically):
\begin{itemize}
    \item When ${\hbeta}^{T_1}_1 > 0$, ${\hbeta}^{T_1}_2 = 0$, the feasible set of Problem~\eqref{eq:refitting_with_ell_1_constrained} is given by
    \begin{equation*}
        \{(\beta_1,\beta_2)^\top \in \bbR \times \bbR : |\beta_1| - \beta_1 + |\beta_2| - \hrho^{T_1}_2\beta_2 \leq T_2 \}\enspace.
    \end{equation*}
    Notice that any positive value for $\beta_1$ is admissible in this case. We illustrate this phenomenon on~\Cref{fig:sparse_lasso_solution_small_bregman}, \Cref{fig:sparse_lasso_solution_large_bregman} and \Cref{fig:T2_is_zero} for various values of $T_2$.
    \item When ${\hbeta}^{T_1}_1 > 0$, ${\hbeta}^{T_1}_2 > 0$, the feasible set of Problem~\eqref{eq:refitting_with_ell_1_constrained} is given by
    \begin{equation*}
        \{(\beta_1,\beta_2)^\top \in \bbR \times \bbR :  |\beta_1| - \beta_1 + |\beta_2| - \beta_2 \leq T_2\}\enspace.
    \end{equation*}
    Notice that every pair $(\beta_1, \beta_2)^\top \in \bbR^{+}\times\bbR^{+}$ is inside of the feasible set.
    We illustrate this phenomenon on~\Cref{fig:not_sparse_lasso_solution}.
\end{itemize}

\begin{figure}[t!]
\centering
        \begin{subfigure}[b]{0.49\textwidth}
        \centering
        \includegraphics[trim={1cm 2cm 1cm 1.5cm},clip,width=0.99\textwidth]{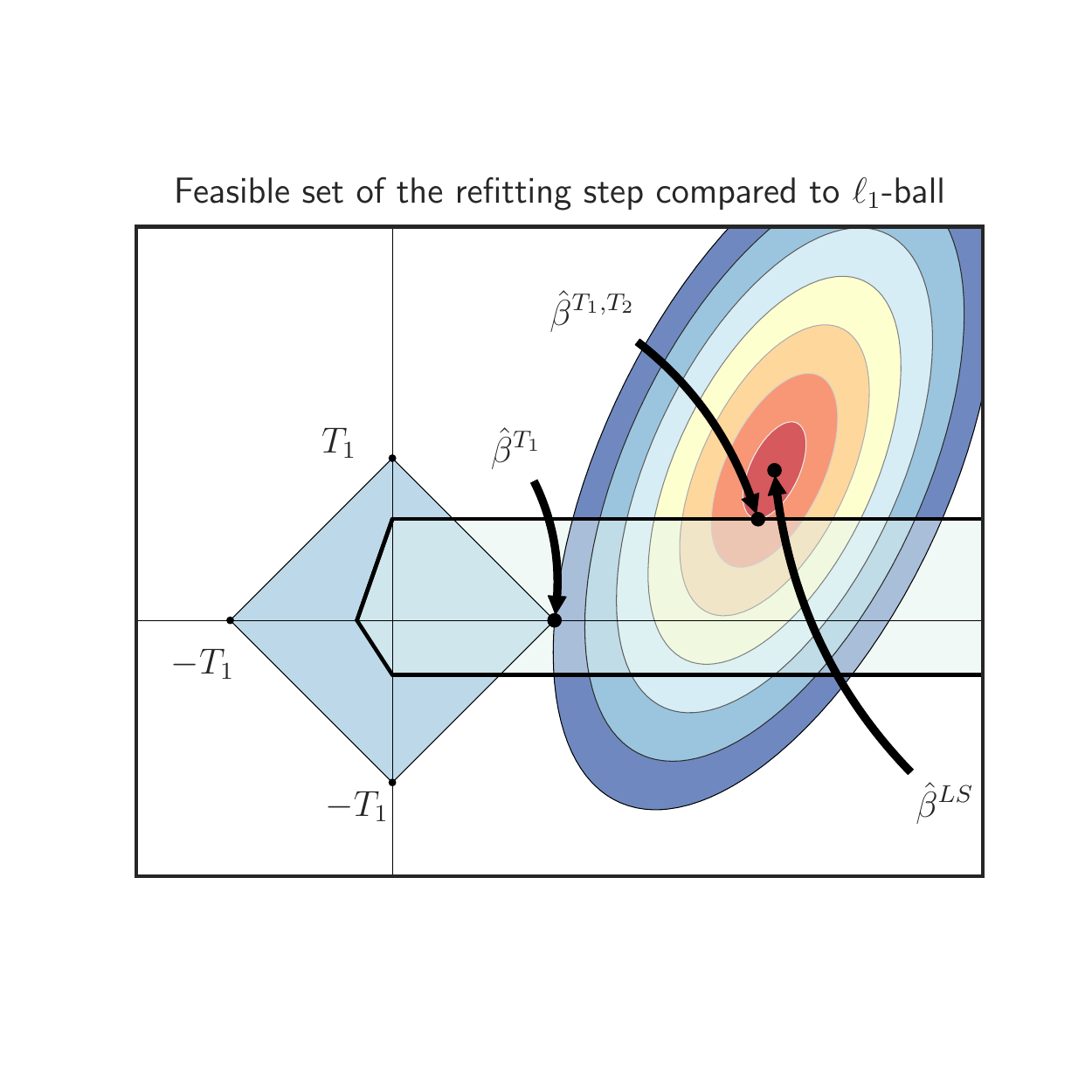}
        \caption{Sparse Lasso solution, $T_2$ - small.}
        \label{fig:sparse_lasso_solution_small_bregman}
        \end{subfigure}
        \begin{subfigure}[b]{0.49\textwidth}
        \centering
        \includegraphics[trim={1cm 2cm 1cm 1.5cm},clip,width=0.99\textwidth]{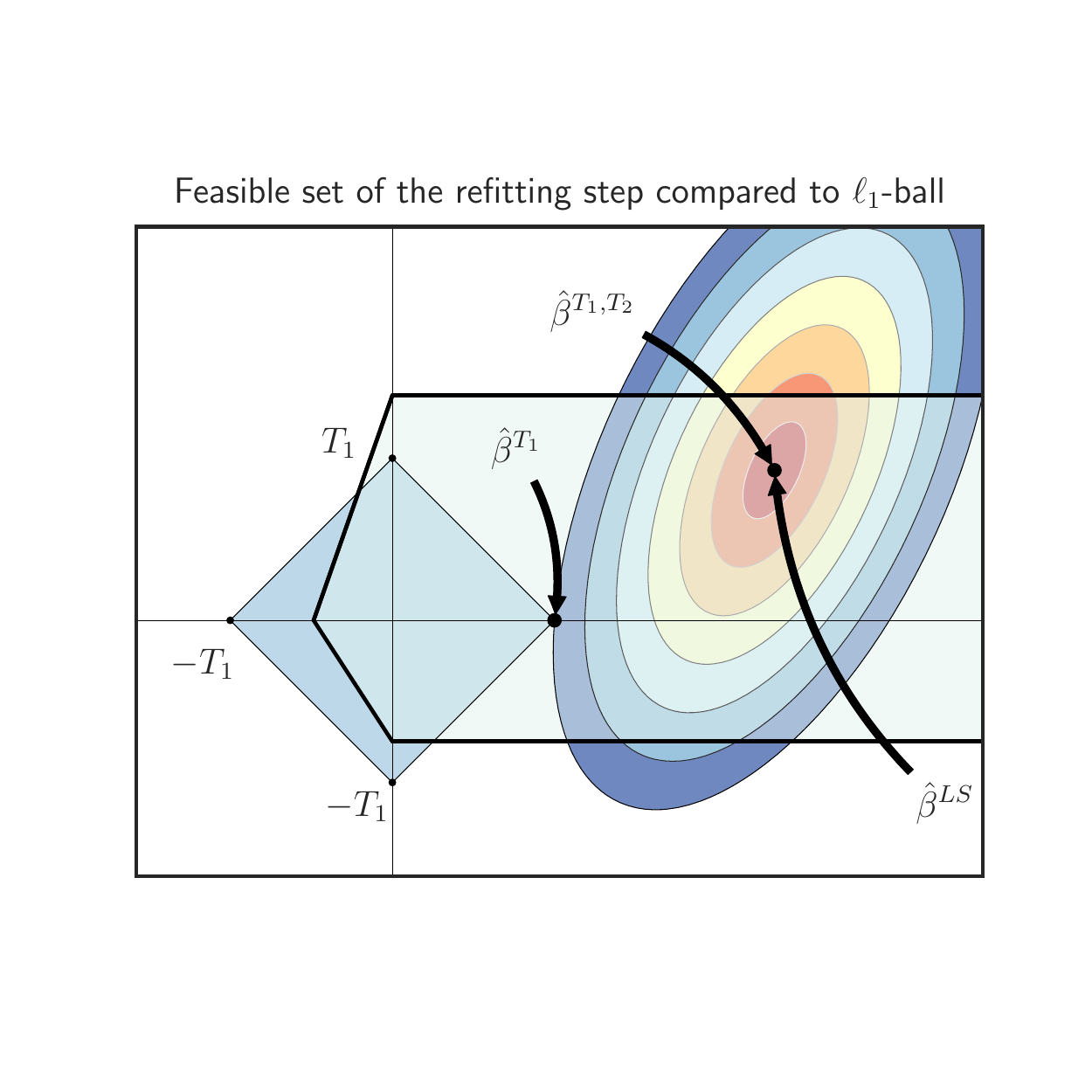}
        \caption{Sparse Lasso solution, $T_2$ - large.}
        \label{fig:sparse_lasso_solution_large_bregman}
        \end{subfigure}
        \begin{subfigure}[b]{0.49\textwidth}
        \centering
        \includegraphics[trim={1cm 2cm 1cm 1.5cm},clip,width=0.99\textwidth]{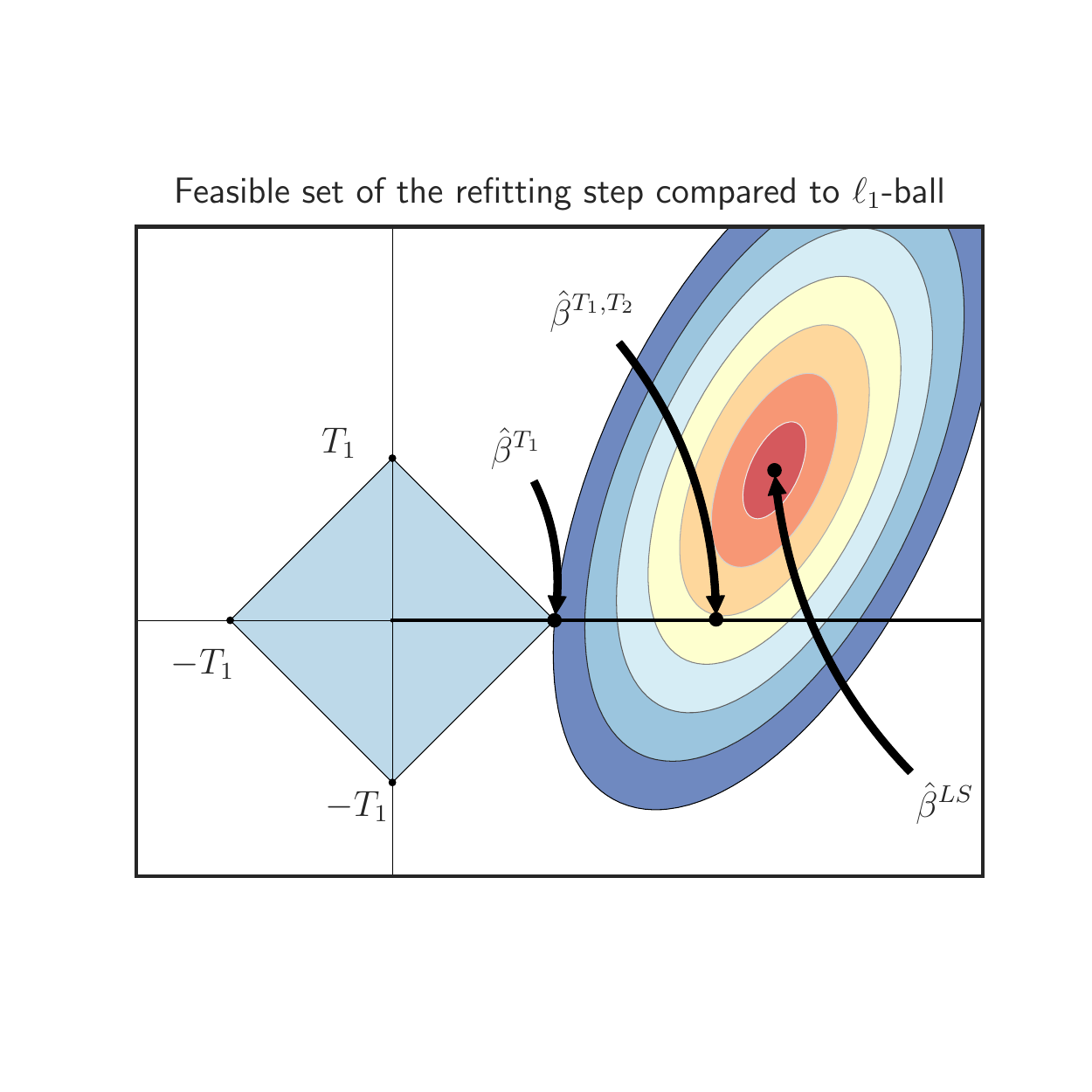}
        \caption{Sparse Lasso solution, $T_2 = 0$.}
        \label{fig:T2_is_zero}
        \end{subfigure}
        \begin{subfigure}[b]{0.49\textwidth}
        \centering
        \includegraphics[trim={1cm 2cm 1cm 1.5cm},clip,width=0.99\textwidth]{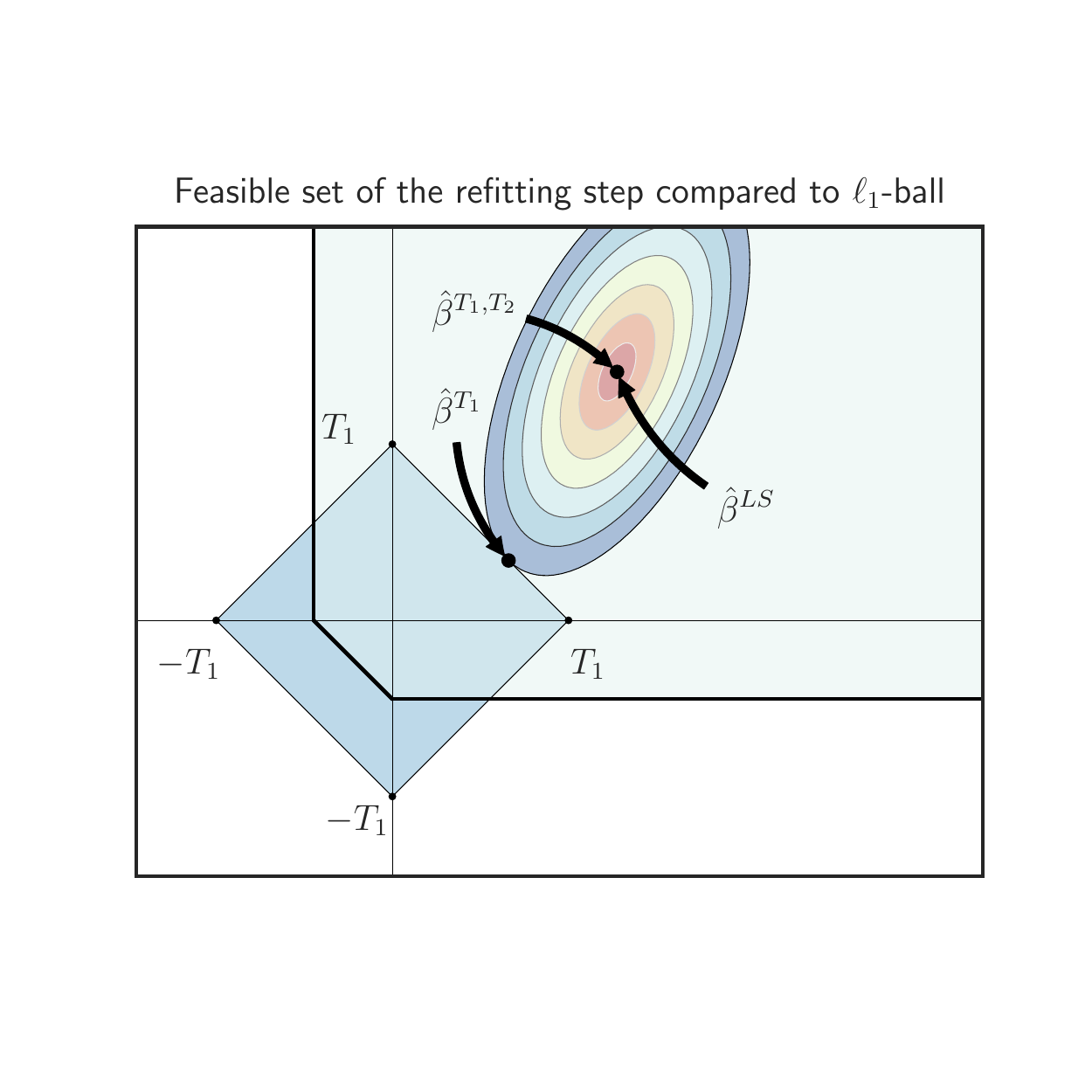}
        \caption{Dense Lasso solution.}
        \label{fig:not_sparse_lasso_solution}
        \end{subfigure}
        \caption{Feasible set of the first and the refitting steps. Here, the ellipses are levels of the objective function and $\hbeta^{LS}$ is the least squares estimator. Geometry of the feasible set for the refitting step is described in terms of $T_2$ and the subgradient $\hrho^{T_1}$.}
\end{figure}
In the Lasso case, if $T_1=0$ there is only one trivial solution $\hbeta^{T_1} = 0$, this is not the case for the refitting step.
Indeed, if $T_2 = 0$ the feasible set of the refitting step might be non-trivial, depending on the first step Lasso solution.
More precisely, if there exists $j_1 \in [p]$, such that $\hbeta^{T_1}_{j_1} \neq 0$ the feasible set of the refitting step always contains $\{\beta : \beta_{j_1} \leq 0, \beta_{j_2} = 0 \ \forall j_2 \neq j_1\}$ or $\{\beta : \beta_{j_1} \geq 0, \beta_{j_2} = 0 \ \forall j_2 \neq j_1\}$, depending on $\sign(\hbeta^{T_1}_{j_1})$.

\subsubsection{Orthogonal design}
\label{sec:linear_regression_with_orthogonal_design}

In this section we
investigate some important properties of the algorithm in~\cref{eq:bregman_iteration_linear} for the denoising model (\ie we drop the statistical convention and instead assume that $n=p$ and $X=I_p$)
$y = \beta^* + \varepsilon$,
where $y, \beta^*  \in \bbR^p$, $\varepsilon \sim \mathcal{N}\big(0, \sigma^2\Id_p\big)$.
The following estimators correspond to~\cref{eq:Lasso} and ~\cref{eq:bregman_iteration_linear} respectively
\begin{align}
    \label{eq:orthogonal_lasso}
    \hbeta^{\lambda_1} &= \argmin_{\beta \in \bbR^p} \frac{1}{2}\norm{y - \beta}_2^2 + \lambda_1\norm{\beta}_1\enspace,\\
    \label{eq:orthogonal_refitting}
    \hbeta^{\lambda_1, \lambda_2} &= \argmin_{\beta \in \bbR^p} \frac{1}{2}\norm{y - \beta}_2^2 + \lambda_2\Big(\norm{\beta}_1 - \scalar{\hrho^{\lambda_1}}{\beta}\Big)\enspace,
\end{align}
where the subgradient $\hrho^{\lambda_1}$ is given by~\cref{eq:refitting_subgradient} and simplifies to
\begin{equation}
    \label{eq:orthogonal_subgradient}
    \hrho^{\lambda_1} = \frac{y - \hbeta^{\lambda_1}}{\lambda_1}\enspace.
\end{equation}
We remind that for orthogonal design, the Lasso estimator is simply a soft-thresholding version of the observation  $\hbeta^{\lambda_1}=\ST(y, \lambda_1)$, where $\ST(\cdot, \cdot)$ is defined component-wise for any $j\in [p]$ by:
 $   \ST(y, \lambda_1)_j = \sign(y_j)(|y_j| - \lambda_1)_+$.

The next proposition shows that the subgradient follows the signs of $y$.
\begin{proposition}
    \label{prop:orthogonal_subgradien_same_sign}
    For all $\lambda_1 > 0$ we have $\sign(y_j) = \sign(\hrho^{\lambda_1}_j),  \forall j \in [p]$.
\end{proposition}
The following results states that the Bregman refitting also enjoys close properties to the Lasso but thresholds a translated response vector.
This relation can be obviously proved using~\Cref{eq:bregman_as_Lasso} of~\Cref{prop:bregman_as_Lasso}, which establishes that the \Bregman can be seen as Lasso solution applied to a modified signal.
\begin{proposition}
    \label{prop:solution_of_refitting_with_orthogonal design}
    The solution of the \Bregman refitting step in~\Cref{eq:orthogonal_refitting} relies on the soft-threshold operator and reads:\, $\hbeta^{\lambda_1, \lambda_2} = \ST(y + \lambda_2 \hrho^{\lambda_1}, \lambda_2)$.
\end{proposition}

It is worth mentioning, that in the Lasso case, there exists a so called $\lambda_{1, \max}$, which is the smallest value of regularization parameter for which the solution $\hbeta^{\lambda_1} = 0$ for all $\lambda_1 > \lambda_{1, \max}$.
However, even though the refitting step can be formulated as Lasso problem, there is not such parameter $\lambda_{2, \max}$.
It can be counter intuitive on the first sight, but since the parameter $\lambda_2$ is present inside the data-fitting term, it becomes clear that such
extreme value does not exist.

\begin{figure}[t!]
\centering
    \begin{subfigure}[b]{0.49\textwidth}
    \centering
    \includegraphics[width=0.99\textwidth]{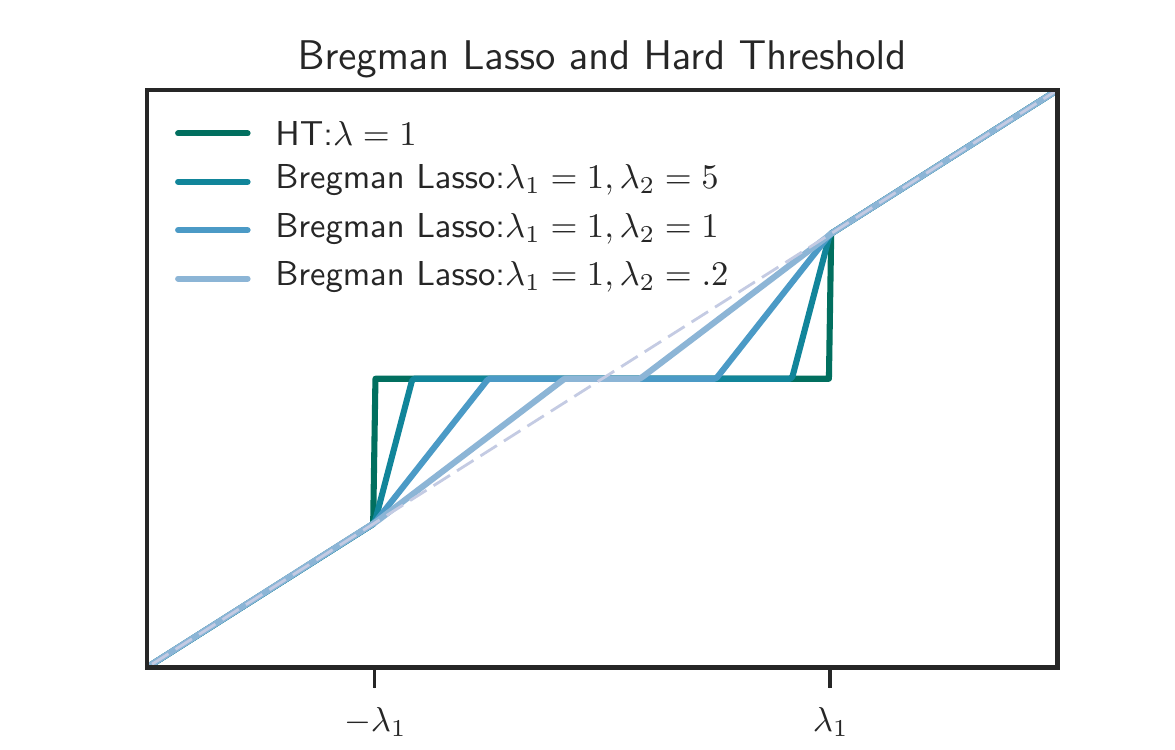}
    \caption{Fixed $\lambda_1$, various $\lambda_2$.}
    \label{fig:2bregman_hard}
    \end{subfigure}
    \begin{subfigure}[b]{0.49\textwidth}
    \centering
    \includegraphics[width=0.99\textwidth]{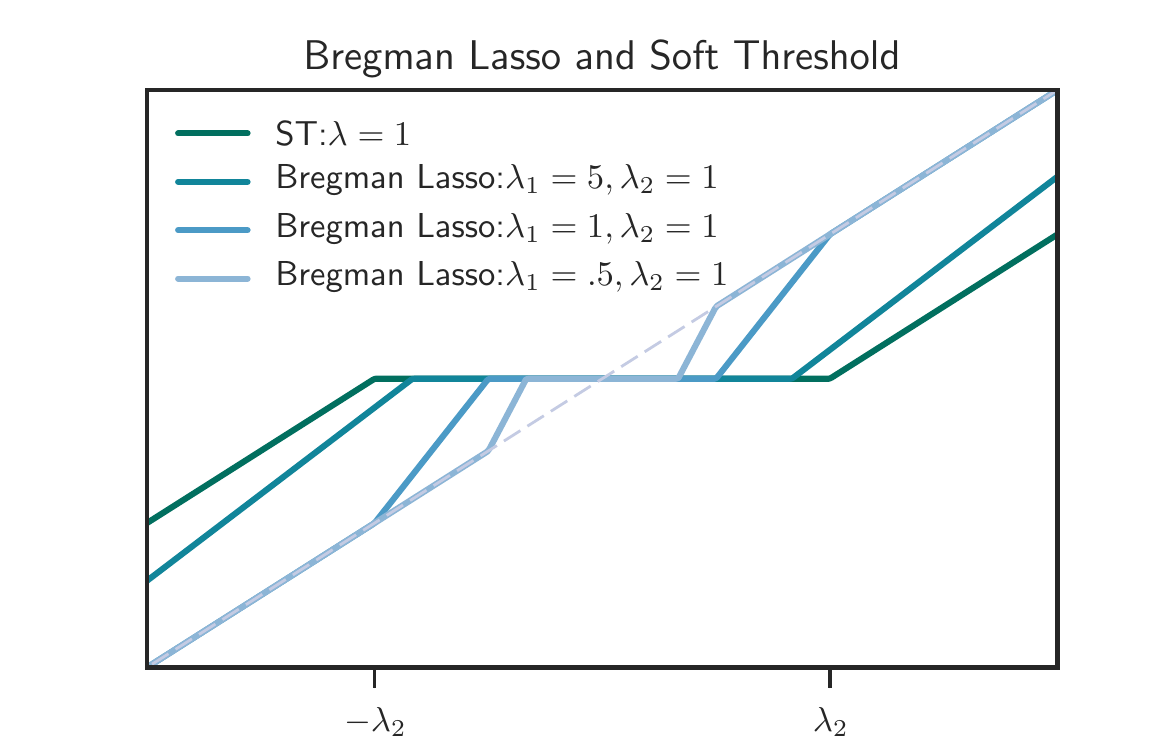}
    \caption{Fixed $\lambda_2$, various $\lambda_1$.}
    \label{fig:2bregman_soft}
    \end{subfigure}
    \caption{The solution of the refitting step is equivalent to the MCP penalty. Extreme cases include both hard and soft threshold operators.}
\end{figure}
In the sequel we provide another interpretation of the \Bregman refitting in terms of firm-thresholding operator, introduced and analyzed by~\citet{Gao_Bruce_97}.
We additionally mention, that the firm-thresholding operator is the solution of the least squares problem penalized with MCP regularization~\citep{Zhang10} (for orthogonal design), which is a non-convex problem.
Additionally, the firm-thresholding operator outperforms soft/hard-threshold in terms of bias-variance trade-off, see~\citep{Gao_Bruce_97} for theoretical and numerical analysis.
\begin{proposition}
    \label{prop:orthogonal_refitting_is_mcp}
    The solution of the \Bregman refitting step in case of orthogonal design is given by: \, $\hbeta^{\lambda_1, \lambda_2} = \MCP(y, \lambda_H, 1 + \tfrac{\lambda_1}{\lambda_2})$,
    where $\lambda_H = (1/\lambda_1 + 1/\lambda_2)^{-1}$ and the firm-thresholding operator $\MCP$ is defined for $\mu > 0$, $\gamma > 1$ component-wise for any $j\in [p]$ by:
    \begin{equation}\label{eq:firm_thresholding}
        \MCP(y, \mu, \gamma)_j = \begin{cases}
                                            \frac{\gamma}{\gamma - 1}\ST(y_j, \mu), \quad &\abs{y_j}\leq \mu \gamma \\
                                            y_j, \quad &\abs{y_j} > \mu \gamma
                                     \end{cases}\enspace.
    \end{equation}
    In the \Bregman (orthogonal) case $\mu = \lambda_H$ and $\gamma = 1 + \lambda_1/\lambda_2=\lambda_1/\lambda_H$.
\end{proposition}

\begin{remark}
    \label{rem:extreme_cases_of_the_MCP_operator}
    We notice the following behavior of the solution, described in~\Cref{prop:orthogonal_refitting_is_mcp}:
    \begin{itemize}
        \item $\hbeta^{\lambda_1, \lambda_2} \rightarrow \HT(y, \lambda_1)$, when $\lambda_2 \rightarrow \infty$,
        \item $\hbeta^{\lambda_1, \lambda_2} \rightarrow \ST(y, \lambda_2)$, when $\lambda_1 \rightarrow \infty$.
    \end{itemize}
    These properties are illustrated on~\Cref{fig:2bregman_hard,fig:2bregman_soft} for several values of $\lambda_1$ and $\lambda_2$.
\end{remark}
{We would like to emphasize that the Bregman Lasso refitting formulated in the orthogonal case, \ie \Cref{eq:firm_thresholding}, coincides with the Firm Thresholding \cite{Gao_Bruce_97}.
In the extreme case where $\lambda_2 \rightarrow \infty$, one recovers the Hard Thresholding. Note that these two cases are connected to non-convex regularizations corresponding to the MCP \cite{Zhang10} and the $\ell_0$.}
For the sake of completeness, we additionally provide the analysis for the Bregman Iterations (orthogonal case) in the form of~\Cref{eq:old_bregman_iteration} to give insights on our motivation to consider two-step iteration with two tuning parameters $\lambda_1, \lambda_2$.
Similar analysis can be found in~\citep{Yin_Osher_Goldfarb_Darbon08, Xu_Osher07}.
Following the proof of~\Cref{prop:solution_of_refitting_with_orthogonal design} one can prove by induction that the following result holds for the Bregman Iterations defined in~\Cref{eq:old_bregman_iteration}.
\begin{proposition}
    For any $\lambda > 0$ and $k > 0$, the Bregman Iteration~\cref{eq:old_bregman_iteration} is given by
    \begin{align*}
        \hbeta_{k+1} = \ST(y + \lambda\hrho_{k}, \lambda)\enspace.
    \end{align*}
\end{proposition}
We also prove that for every $k > 1$, firm-thresholding yields to Bregman Iterations.
\begin{proposition}\label{prop:breg_isfirw}
    For any $\lambda > 0$ and $k > 0$, the solution generated by the Bregman Iterations~\cref{eq:old_bregman_iteration} is given by
    \begin{align*}
        \hbeta_{k+1} = \MCP(y, \tfrac{\lambda}{k+1}, \tfrac{k+1}{k}) = \begin{cases}
        (k+1)\ST(y, \tfrac{\lambda}{k+1}), &\text{ if } |y_j| \leq \tfrac{\lambda}{k}\\
        y_j&\text{ if } |y_j| > \tfrac{\lambda}{k}
        \end{cases}\enspace.
    \end{align*}
\end{proposition}
Analyzing the previous result helps motivating the introduction of the \Bregman refitting in the form of \cref{eq:bregman_iteration_linear}.
Indeed, for orthogonal design, \Bregman refitting generalizes Bregman Iterations~\cref{eq:old_bregman_iteration} in the sense that for each pair $(k, \lambda)$ there exists a pair $(\lambda_1 = \tfrac{\lambda}{k-1}, \lambda_2=\lambda)$ such that $\hbeta_k = \hbeta^{\lambda_1, \lambda_2}$. In particular, \Bregman refitting includes all possible solution of Bregman Iterations~\cref{eq:old_bregman_iteration}, while the reverse is not true, for instance when $k$ is not an integer.

\section{\SLSLasso}
\label{sec:arbitrary_design}
In this section we provide some generalizations of the results obtained in~\Cref{sec:linear_regression_with_orthogonal_design}.
We first notice that in the case of orthogonal design, discussed in~\Cref{sec:linear_regression_with_orthogonal_design}, for a given vector $y$, there exists a value of the regularization parameter $\lambda_2$ such that, the refitting step~\cref{eq:bregman_iteration_linear} is equivalent to hard-thresholding.
One might expect that for the case of arbitrary design there exists such a value of $\lambda_2$ that the refitting step~\cref{eq:bregman_iteration_linear} is equivalent to a Least-Squares refitting on the support obtained via the first Lasso type step as it adds a sign constraints.
Yet, the estimator obtained via the refitting step slightly differs from the simple Least-Squares refitting.
To state our main result of this section, let us introduce some notation.
Consider $\hrho^{\lambda_1}$, defined in~\cref{eq:refitting_subgradient}, we define the equicorrelation set~\citep{Tibshirani13} as
\begin{align}
    \label{eq:equicorrelation_set}
    E^{\lambda_1} = E(\hrho^{\lambda_1}) = \{j \in [p] : |\hrho^{\lambda_1}_j| = 1\}\enspace.
\end{align}
We will omit $\lambda_1$ in $E^{\lambda_1}$ and write $E$ instead for simplicity.
We associate the following \SLSLasso refitting step with the equicorrelation set given by~\cref{eq:equicorrelation_set}
\begin{definition}[\SLSLasso]
For any $\lambda_1 > 0$ we call $\slsbeta$ a \SLSLasso refitting if it satisfies $\slsbeta_{E^{c}} = 0$ and
\begin{align}
  \label{eq:sls_estim}
    \slsbeta_{E} \in \argmin_{\beta_{E} \in \bbR^{|E|} : \hrho^{\lambda_1}_{E} \odot \beta_{E} \geq 0} \frac{1}{2n}\normin{y - X_E\beta_E}_2^2\enspace,
\end{align}
where $|E|$ is the cardinality of $E$ and $\hrho^{\lambda_1}$ is the Lasso subgradient defined in~\Cref{eq:refitting_subgradient}.
\end{definition}
Notice, that the refitting is performed on the equicorrelation set and not on the support of the Lasso solution.
Possible motivation to consider the equicorrelation set instead of the support can be described by uniqueness issues of the Lasso~\citep{Tibshirani13}, while the equicorrelation set (and the signs) is always uniquely defined.
\begin{proposition}
    The \SLSLasso is a sign consistent refitting strategy of the Lasso solution $\hbeta^{\lambda_1}$ in the sense of~\Cref{def:sign_refitting_vector}.
\end{proposition}
\begin{proof}
    Notice that we have the relation
    $\{\beta_{E} \in \bbR^{|E|} : \hrho^{\lambda_1}_{E} \odot \beta_{E} \geq 0\} = \{\beta_{E} \in \bbR^{|E|} : \hrho^{\lambda_1} \in \sign(\beta)\}$,
    and using~\Cref{rem:sign_refitting} we conclude.
\end{proof}
{ We believe that the introduced Sign-Least-Squares refitting of the Lasso is an interesting and simple alternative to the Least-Squares Lasso.
Indeed, both methods achieve similar theoretical guarantees and their computational costs are equivalent (since the estimated supports are generally small).
Moreover, the Sign-Least-Squares refitting of the Lasso might be more appealing to practitioners since confusing signs switches are no longer possible.}
\subsection{\Bregman as \SLSLasso}
\label{subsec:bregman_as_sls}

In this section we provide connections between the \Bregman and the \SLSLasso.
Since Problem~\eqref{eq:sls_estim} is convex and Slater's condition is satisfied therefore the KKT conditions state~\citep{Boyd_Vandenberghe04} that there exist $\slsbeta_{E} \in \bbR^{|E|} \text{ and } \hat{\mu} \in \bbR^{|E|}_+$, such that
\begin{equation}
\label{eq:KKT}
    \begin{cases}
        \frac{1}{n}X^\top_{E}\big(y - X_{E}\slsbeta_{E}\big) + \hrho^{\lambda_1}_{E} \odot \hat{\mu} = 0 \\
        \hat{\mu} \odot \slsbeta_{E} \odot \hrho^{\lambda_1}_{E} = 0\\
        \hat{\mu} \geq 0
    \end{cases}\enspace,
\end{equation}
so that we can provide the following connection:
\begin{theorem}
\label{th:SLS}
 For any signal vector $y \in \bbR^n$ and any design matrix $X \in \bbR^{n \times p}$ there exists $\lambda_0 > 0$ such that for any $\lambda_2 \geq \lambda_0$ the solution of the refitting step~\cref{eq:bregman_iteration_linear} is given by  $\hbeta^{\lambda_1, \lambda_2} = \slsbeta$.
\end{theorem}
The proof the this theorem is below. However we mention that it states that given $\lambda_1>0$, \Bregman and \SLSLasso often coincide. The value of the parameter $\lambda_0$ in \Cref{th:SLS} is explicit and can be found at the end of the following proof.

\begin{proof}
    We start by writing the KKT conditions for \Bregman,~\cref{eq:bregman_iteration_linear}:
    \begin{equation}
    \label{eq:fermat}
        \begin{cases}
            \hrho^{\lambda_1}_{E} + \frac{1}{\lambda_2 n}X_{E}^\top \big(y - X_{E}\hbeta^{\lambda_1, \lambda_2}_{E} - X_{E^c} \hbeta^{\lambda_1, \lambda_2}_{E^{c}}\big) \in \sign(\hbeta^{\lambda_1, \lambda_2}_{E}) \\
            \hrho^{\lambda_1}_{E^c} + \frac{1}{\lambda_2 n}X_{E^c}^\top \big(y - X_{E}\hbeta^{\lambda_1, \lambda_2}_E - X_{E^c} \hbeta^{\lambda_1, \lambda_2}_{E^c}\big) \in \sign(\hbeta^{\lambda_1, \lambda_2}_{E^c})
        \end{cases}
        \enspace.
    \end{equation}
    These are necessary and sufficient conditions for a vector $\beta$ to be a solution of \Bregman~\eqref{eq:bregman_iteration_linear}.
    Therefore, it is sufficient to check if the vector defined as $\hbeta^{\lambda_1, \lambda_2} = \slsbeta$ satisfies \cref{eq:fermat} for some $\lambda_2$.
    Substituting $\hbeta^{\lambda_1, \lambda_2}_{E} = \slsbeta_{E}$ and $\hbeta^{\lambda_1, \lambda_2}_{E^c} = 0$ we arrive at
        \begin{equation}
        \label{eq:fermat_modif}
        \begin{cases}
            \hrho^{\lambda_1}_{E} + \frac{1}{\lambda_2 n}X_{E}^\top \big(y - X_{E}\slsbeta_{E}\big) \in \sign(\slsbeta_{E})\\
            \hrho^{\lambda_1}_{E^c} + \frac{1}{\lambda_2 n}X_{E^c}^\top \big(y - X_{E}\slsbeta_{E}\big) \in \sign(0)
        \end{cases}
        \enspace.
    \end{equation}
    Since $\slsbeta_{E}$ satisfies the KKT conditions, given in~\cref{eq:KKT}, this yields to
    \begin{equation}
    \label{eq:fermat_modif2}
        \begin{cases}
            \hrho^{\lambda_1}_{E} \odot(\1_{|E|} - \frac{1}{\lambda_2}\hat{\mu}) \in \sign(\slsbeta_{E})\\
            \hrho^{\lambda_1}_{E^c} + \frac{1}{\lambda_2 n}X_{E^c}^\top \big(y - X_{E}\slsbeta_{E}\big) \in \sign(0)
        \end{cases}
        \enspace,
    \end{equation}
    where $\1_{|E|} = {(1, \ldots, 1)^\top}\in\bbR^{|E|}$.
    We notice that the second line in~\cref{eq:fermat_modif2} can be satisfied for $\lambda_2$ large enough, since for each $j \in E^c$ we have $|\hrho^{\lambda_1}_j| < 1$ and $\slsbeta_{E}$ does not depend on the value of $\lambda_2$.
    Denote $\lambda$ the smallest $\lambda_2$ such that the second condition in~\cref{eq:fermat_modif2} is satisfied.
    Now, notice that the first line can be studied element-wise, hence we study two distinct cases, where we take $j\in E$.
    \begin{itemize}
            \item $\hat{\mu}_j = 0 \implies \hrho^{\lambda_1}_j \in \sign\big(\slsbeta_{E}\big)_j$, which holds due to the definition of $\slsbeta_{E}$,
            \item $\hat{\mu}_j > 0 \overset{\eqref{eq:KKT}}{\implies} \big(\slsbeta_{E}\big)_j = 0 \implies \text{ if } \lambda_2 \geq \hat{\mu}_j $ then $ |\hrho^{\lambda_1}_j(1 - \frac{\hat{\mu}_j}{\lambda_2})| \leq 1$ and \cref{eq:fermat_modif2} is satisfied.
    \end{itemize}
    Setting $\lambda_0 = \lambda \vee  \hat{\mu}_{\max}$, where $\hat\mu_{\max} = \max_{j=1}^{|E|}\{\hat \mu_{j}\}$ concludes the proof.
\end{proof}
Combining~\Cref{thm:bregman_racle,th:SLS} we can state the following corollary, which provides an oracle inequality for the \SLSLasso:
\begin{cor}
    \label{cor:slslasso_oracle}
    {
    If the design matrix $X \in \bbR^{n \times p}$ satisfies the restricted eigenvalue condition RE($3$, $s$) and ${\lambda} = 5\sigma\sqrt{{2\log{(p/\delta)}}/{n}}$ for some $\delta \in (0, 1)$, then with probability $1 - \delta$ the following bound holds:
  \[
  \frac{1}{n}\normin{X(\tbeta - \bbeta)}_2^2 + \frac{1}{n}\normin{X(\tbeta - \hbeta)}_2^2
      \leq \frac{\lambda^2 s}{\kappa^2(1, s)} + \frac{9\lambda^2 s}{4\kappa^2(3, s)}\enspace,
    \]
    where $\bbeta$ is the \SLSLasso solution associated with $\hbeta = \hbeta^{\lambda}$.}
\end{cor}


\begin{figure}[t!]
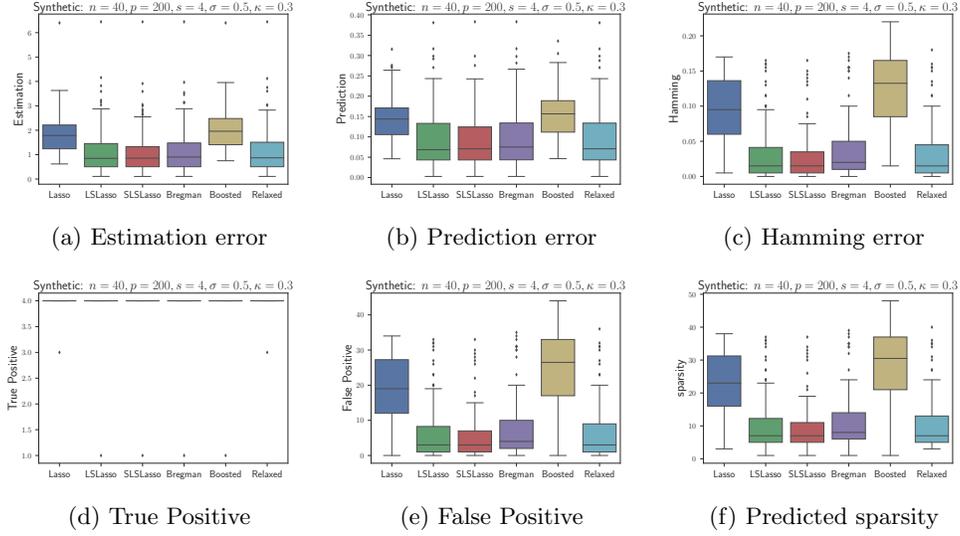

\centering
        \begin{subfigure}[b]{0.32\textwidth}
            \centering
            \includegraphics[width=0.99\textwidth]{{{"code/graphs/n:40p:200sigma:0.5-corr:0.3s:4/Estimation"}}}
            \caption{Estimation error}
        \end{subfigure}
        \begin{subfigure}[b]{0.32\textwidth}
            \centering
            \includegraphics[width=0.99\textwidth]{{{"code/graphs/n:40p:200sigma:0.5-corr:0.3s:4/Prediction"}}}
            \caption{Prediction error}
        \end{subfigure}
        \begin{subfigure}[b]{0.32\textwidth}
            \centering
            \includegraphics[width=0.99\textwidth]{{{"code/graphs/n:40p:200sigma:0.5-corr:0.3s:4/Hamming"}}}
            \caption{Hamming error}
        \end{subfigure}
        \begin{subfigure}[b]{0.32\textwidth}
            \centering
            \includegraphics[width=0.99\textwidth]{{{"code/graphs/n:40p:200sigma:0.5-corr:0.3s:4/True_Positive"}}}
            \caption{True Positive}
        \end{subfigure}
        \begin{subfigure}[b]{0.32\textwidth}
            \centering
            \includegraphics[width=0.99\textwidth]{{{"code/graphs/n:40p:200sigma:0.5-corr:0.3s:4/False_Positive"}}}
            \caption{False Positive}
        \end{subfigure}
        \begin{subfigure}[b]{0.32\textwidth}
            \centering
            \includegraphics[width=0.99\textwidth]{{{"code/graphs/n:40p:200sigma:0.5-corr:0.3s:4/sparsity"}}}
            \caption{Predicted sparsity}
        \end{subfigure}
        \caption{Synthetic dataset: low correlation scenario.}
        \label{plot:synth_low_corr}
\end{figure}

\begin{figure}[t!]
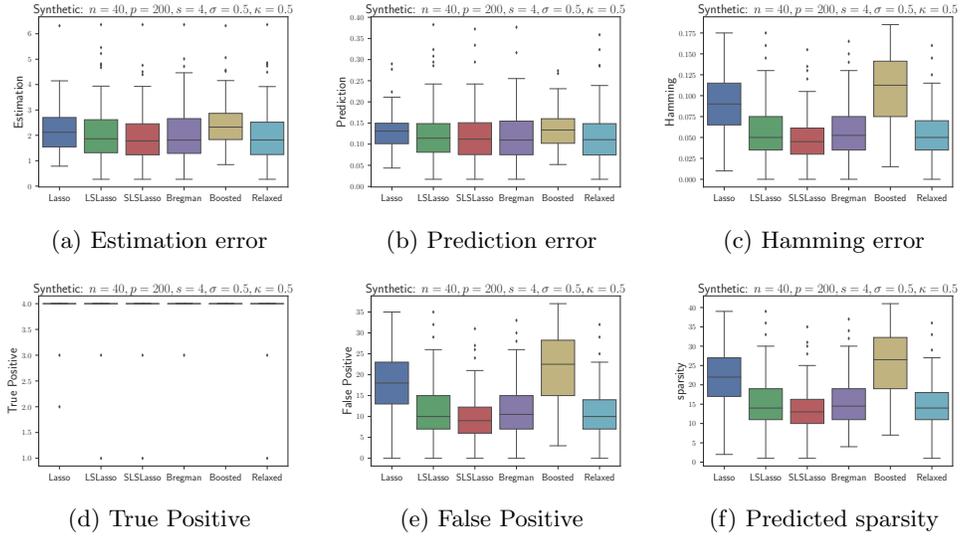

\centering
        \begin{subfigure}[b]{0.32\textwidth}
            \centering
            \includegraphics[width=0.99\textwidth]{{{"code/graphs/n:40p:200sigma:0.5-corr:0.5s:4/Estimation"}}}
            \caption{Estimation error}
        \end{subfigure}
        \begin{subfigure}[b]{0.32\textwidth}
            \centering
            \includegraphics[width=0.99\textwidth]{{{"code/graphs/n:40p:200sigma:0.5-corr:0.5s:4/Prediction"}}}
            \caption{Prediction error}
        \end{subfigure}
        \begin{subfigure}[b]{0.32\textwidth}
            \centering
            \includegraphics[width=0.99\textwidth]{{{"code/graphs/n:40p:200sigma:0.5-corr:0.5s:4/Hamming"}}}
            \caption{Hamming error}
        \end{subfigure}
        \begin{subfigure}[b]{0.32\textwidth}
            \centering
            \includegraphics[width=0.99\textwidth]{{{"code/graphs/n:40p:200sigma:0.5-corr:0.5s:4/True_Positive"}}}
            \caption{True Positive}
        \end{subfigure}
        \begin{subfigure}[b]{0.32\textwidth}
            \centering
            \includegraphics[width=0.99\textwidth]{{{"code/graphs/n:40p:200sigma:0.5-corr:0.5s:4/False_Positive"}}}
            \caption{False Positive}
        \end{subfigure}
        \begin{subfigure}[b]{0.32\textwidth}
            \centering
            \includegraphics[width=0.99\textwidth]{{{"code/graphs/n:40p:200sigma:0.5-corr:0.5s:4/sparsity"}}}
            \caption{Predicted sparsity}
        \end{subfigure}
        \caption{Synthetic dataset: average correlation scenario.}
        \label{plot:synth_average_corr}
\end{figure}


\section{Experiments}
\label{sec:experiments}
To evaluate each scenario, we consider the following 'oracle' performance measures (in the sense that in practice one can not evaluate them):
\begin{itemize}
    \item Prediction: $\normin{X(\tbeta - \hbeta)}_2^2$;
    \item Estimation: $\normin{\tbeta - \hbeta}_1$;
    \item Predicted sparsity: $|\supp(\hbeta)|$;
    \item True Positive: $|\{j\in[p]: \tbeta_j \neq 0 \text{ and } \hbeta_j \neq 0\}|$;
    \item False Positive: $|\{j\in[p]: \tbeta_j = 0 \text{ and } \hbeta_j \neq 0\}|$;
    \item Hamming: $\frac{1}{p}|\{j\in[p]: \text{False Positive} \text{ or } \text{False Negative} \text{ or } \text{False Sign}\}|$,

    where $\underbrace{\tbeta_j = 0 \text{ and } \hbeta_j \neq 0}_{\text{False Positive}}, \ \underbrace{\tbeta_j \neq 0 \text{ and } \hbeta_j = 0}_{\text{False Negative}},  \ \underbrace{\sign(\tbeta_j) \neq \sign(\hbeta_j)}_{\text{False Sign}}$.
\end{itemize}

\begin{figure}[t!]
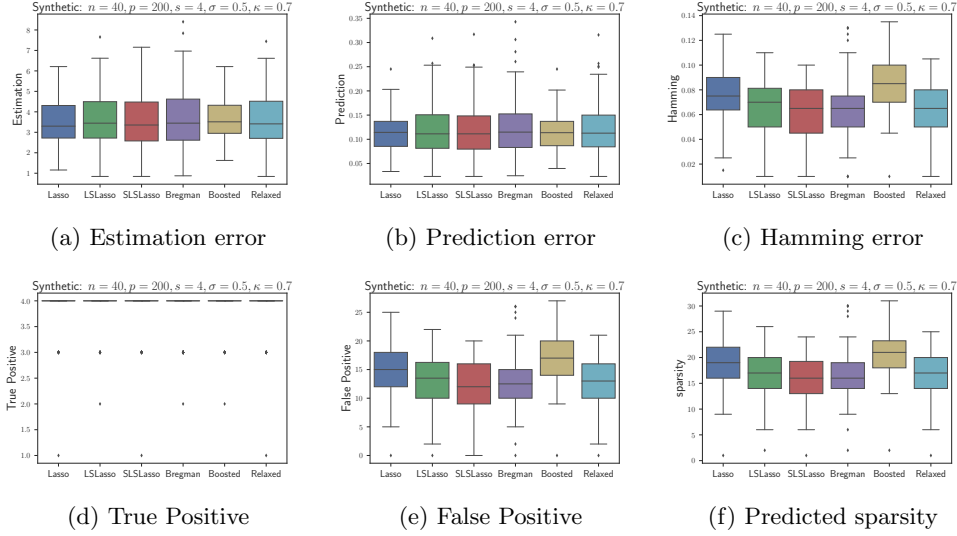

\centering
        \begin{subfigure}[b]{0.32\textwidth}
            \centering
            \includegraphics[width=0.99\textwidth]{{{code/graphs/n:40p:200sigma:0.5-corr:0.7s:4/Estimation}}}
            \caption{Estimation error}
        \end{subfigure}
        \begin{subfigure}[b]{0.32\textwidth}
            \centering
            \includegraphics[width=0.99\textwidth]{{{code/graphs/n:40p:200sigma:0.5-corr:0.7s:4/Prediction}}}
            \caption{Prediction error}
        \end{subfigure}
        \begin{subfigure}[b]{0.32\textwidth}
            \centering
            \includegraphics[width=0.99\textwidth]{{{code/graphs/n:40p:200sigma:0.5-corr:0.7s:4/Hamming}}}
            \caption{Hamming error}
        \end{subfigure}
        \begin{subfigure}[b]{0.32\textwidth}
            \centering
            \includegraphics[width=0.99\textwidth]{{{"code/graphs/n:40p:200sigma:0.5-corr:0.7s:4/True_Positive"}}}
            \caption{True Positive}
        \end{subfigure}
        \begin{subfigure}[b]{0.32\textwidth}
            \centering
            \includegraphics[width=0.99\textwidth]{{{"code/graphs/n:40p:200sigma:0.5-corr:0.7s:4/False_Positive"}}}
            \caption{False Positive}
        \end{subfigure}
        \begin{subfigure}[b]{0.32\textwidth}
            \centering
            \includegraphics[width=0.99\textwidth]{{{code/graphs/n:40p:200sigma:0.5-corr:0.7s:4/sparsity}}}
            \caption{Predicted sparsity}
        \end{subfigure}
        \caption{Synthetic dataset: high correlation scenario.}
        \label{plot:synth_high_corr}
\end{figure}
All except one measures used in our evaluation are standard and are widely considered by various authors.
We additionally study Hamming loss, which allows to capture information about miss predicting the sign of the underlying $\tbeta$.
However, one should keep in mind that the introduced Hamming treats False Positive, False Negative and False Sign mistakes equally.
Furthermore, the estimators that we consider are described in~\Cref{tab:all_estimatros}.

\begin{table}[h!]
\centering
\begin{footnotesize}
\renewcommand{\arraystretch}{2}
\begin{tabular}{l|l}
Lasso~\citep{Tibshirani96}& $\hbeta^{Lasso} \in \displaystyle\argmin_{\beta \in \bbR^p} \frac{1}{2n}\normin{y - X\beta}_2^2 + \lambda_1\normin{\beta}_1$\\
\hline
LSLasso~\citep{Belloni_Chernozhukov13} & $\hbeta^{LSLasso} \in \displaystyle\argmin_{\beta : \supp(\beta) \subset \supp(\hbeta^{Lasso})} \frac{1}{2n}\normin{y - X\beta}_2^2$\\
\hline
SLSLasso~\citep{Brinkmann_Burger_Rasch_Satour16} & $\hbeta^{SLSLasso} \in \displaystyle\argmin_{\beta:\beta \odot \hrho^{\lambda_1} \geq 0} \frac{1}{2n}\normin{y - X\beta}_2^2$\\
\hline
Boosted Lasso & $\hbeta^{Boosted} \in \displaystyle\argmin_{\beta \in \bbR^p} \frac{1}{2n}\normin{y - X\beta}_2^2 + \lambda_2\normin{\beta - \hbeta^{Lasso}}_1$\\
\hline
Bregman Lasso & $\hbeta^{Bregman} \in \displaystyle\argmin_{\beta \in \bbR^p} \frac{1}{2n}\normin{y - X\beta}_2^2 + \lambda_2(\normin{\beta}_1 - \scalar{\hrho^{\lambda_1}}{\beta})$\\
\hline
{Relaxed Lasso} & {$\hbeta^{Relaxed} \in \displaystyle\argmin_{\beta : \supp(\beta) \subset \supp(\hbeta^{Lasso})} \frac{1}{2n}\norm{y - X\beta}_2^2 + \phi\lambda_1{\norm{\beta}_1}$}\\
\end{tabular}
\end{footnotesize}
\caption{Description of considered estimators in the experiments. The subgradient $\hrho^{\lambda_1}$, associated with the Lasso solution, is evaluated according to~\Cref{eq:refitting_subgradient} with $\hbeta^{\lambda_1} = \hbeta^{Lasso}$. {For the Relaxed Lasso, the parameter $
\phi$ belongs to the interval $(0, 1)$.}}
\label{tab:all_estimatros}
\end{table}

To report our results, we present boxplots for each scenario and performance measure over one hundred experiment replicas.
During each simulation run we perform 3 fold cross-validation over a predefined one dimensional grid of 50 points, spread equally on logarithmic scale from $0.01 \cdot \norm{X^\top y/n}_{\infty}$ to $\norm{X^\top y/n}_{\infty}$ for $\lambda_1$, or two dimensional grid of $50 \times 50$ points for $\lambda_1, \lambda_2$ on the same scale.
{For the Relaxed Lasso the parameter $\phi$ is chosen over a uniform grid of $50$ points laying in the interval $(0.001, 0.999)$.}
We finally select the tuning parameters achieving the best 3-fold cross-validation performance in terms of MSE, since in practice one does not have access to the underlying $\tbeta$.
\subsection{Synthetic data}
\label{subsec:synthetic_data}

In our synthetic experiments, we generated the design matrix $X \in \bbR^{n \times p}$ as follows
\begin{align}
    X_j = \sqrt{n}\frac{\kappa \zeta + (1-\kappa)\xi_j}{\norm{\kappa \zeta + (1-\kappa)\xi_j}_2}\enspace,
\end{align}
where $\zeta, \xi_1, \ldots, \xi_p \sim \mathcal{N}(0, I_p)$ are independent standard normal vectors.
The level of correlations between the covariates is determined by the parameter $\kappa \in [0, 1]$.
We additionally consider a noise vector $\varepsilon \sim \mathcal{N}(0, I_n)$ and set the underlying vector as
\[
\tbeta = (\underbrace{1,\ldots, 1}_{s}, \underbrace{0, \ldots, 0}_{p - s})^\top \in \bbR^p\enspace.
\]

\begin{figure}[t!]
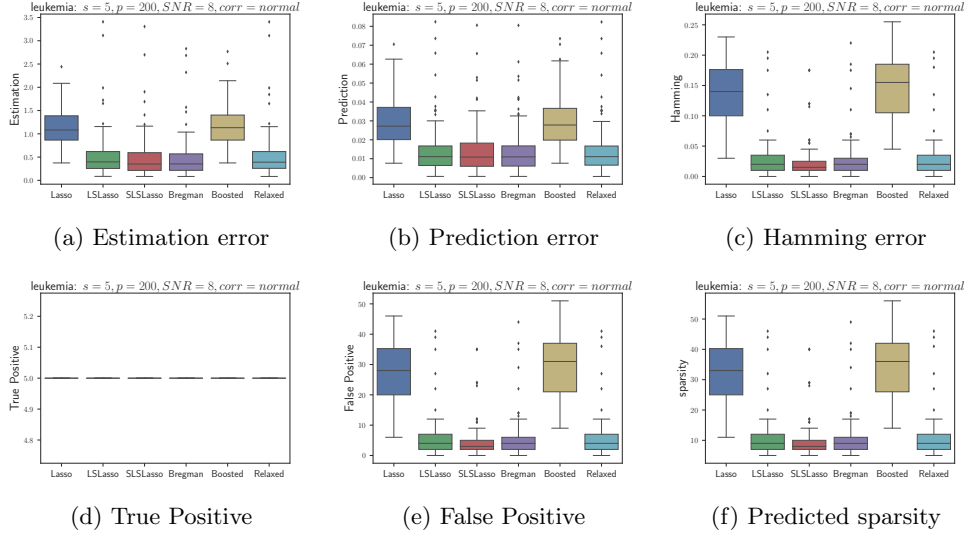

\centering
        \begin{subfigure}[b]{0.32\textwidth}
            \centering
            \includegraphics[width=0.99\textwidth]{{{code/graphs/s:5p:200SNR:8-normal/Estimation}}}
            \caption{Estimation error}
        \end{subfigure}
        \begin{subfigure}[b]{0.32\textwidth}
            \centering
            \includegraphics[width=0.99\textwidth]{{{code/graphs/s:5p:200SNR:8-normal/Prediction}}}
            \caption{Prediction error}
        \end{subfigure}
        \begin{subfigure}[b]{0.32\textwidth}
            \centering
            \includegraphics[width=0.99\textwidth]{{{code/graphs/s:5p:200SNR:8-normal/Hamming}}}
            \caption{Hamming error}
        \end{subfigure}
        \begin{subfigure}[b]{0.32\textwidth}
            \centering
            \includegraphics[width=0.99\textwidth]{{{"code/graphs/s:5p:200SNR:8-normal/True_Positive"}}}
            \caption{True Positive}
        \end{subfigure}
        \begin{subfigure}[b]{0.32\textwidth}
            \centering
            \includegraphics[width=0.99\textwidth]{{{"code/graphs/s:5p:200SNR:8-normal/False_Positive"}}}
            \caption{False Positive}
        \end{subfigure}
        \begin{subfigure}[b]{0.32\textwidth}
            \centering
            \includegraphics[width=0.99\textwidth]{{{code/graphs/s:5p:200SNR:8-normal/sparsity}}}
            \caption{Predicted sparsity}
        \end{subfigure}
        \caption{\emph{Leukemia} dataset with $p=200$, $\SNR=8$, $s=5$, normal correlations scenario.}
        \label{plot:leuk_200_8_5}
\end{figure}
Hence, each scenario in the synthetic data section is described by the following parameters: number of observations: $n$, number of features: $p$, sparsity level of $\tbeta$: $s$, level of correlations: $\kappa$, noise level: $\sigma$.
We fix $n = 40, p = 200, s = 4, \sigma = 0.5$ and use the following values of correlations $\kappa = 0.3; 0.5; 0.7$, representing low, average and high correlation settings respectively.
The results are reported on~\Cref{plot:synth_low_corr},~\Cref{plot:synth_average_corr} and~\Cref{plot:synth_high_corr}.
We first notice that the \Boosted does not give any significant improvement over the Lasso.
Other four refitting strategies can improve the first step Lasso solution.
In case of modest correlations inside the design matrix, the improvement can be considered as significant.
Moreover, the \SLSLasso outperforms the \LSLasso and the \Bregman in average, in addition to the greater interpretability due to the sign preserving properties.
{The Relaxed Lasso performance  is on par with the one of the \LSLasso. However, the former requires an additional tuning parameter to be calibrated.}

\subsection{Semi-real data}
\label{subsec:semi_real_data}
\begin{table}[t!]
\centering
\begin{tabular}{lll}
\hline
\multicolumn{3}{l}{Settings}                   \\ \hline
Signal to noise ration (SNR) & $2$    & $8$    \\
Number of covariates ($p$)   & $200$  & $1000$ \\
Underlying sparsity ($s$)    & $5$    & $20$   \\
Correlations settings        & Normal & High   \\ \hline
\end{tabular}
\caption{Description of the parameters for semi-real dataset simulation}
\label{tab:params_description}
\end{table}
For our experimental study, we generate semi-real datasets, following an approach presented in~\citep{Buhlmann_Mandozzi14}.
The data are generated from the model~\Cref{eq:linear_regression_model}, where the design matrix $X$ is obtained from the \emph{leukemia} dataset with $n = 72$.
We consider the following $4$ parameters to describe the settings of our experiments: $p$ - number of covariates, $s$ - sparsity of the vector $\tbeta$, SNR - signal to noise ratio, and the correlation settings - way to generate support of the vector $\tbeta$.
All the plots are averaged over hundred runs of the simulation process with fixed values of parameters given in~\Cref{tab:params_description}.
During each simulation round we choose the first $p$ columns from the \emph{leukemia} dataset.
Additionally we set the signal to noise ratio, defined as $\SNR = \frac{\norm{X\tbeta}_2}{\sqrt{n\sigma^2}}$,
to control the noise level $\sigma^2$ in the model~\Cref{eq:linear_regression_model}.
Moreover the true cardinality of $\tbeta$ is set to $s$ and each non-zero component of $\tbeta$ is set to $1$ or $-1$ with equal probabilities.
Finally, the support of the vector $\tbeta$ is formed following two scenarios: normal correlations and high correlations.
For normal correlations we choose randomly $s$ components out of $p$ and for high correlations scenario, the first component is chosen randomly and the remaining $s - 1$ having the highest Pearson correlation with the first one.
Additional scenarios can be found in Appendix, in the main text we provide only three cases due to the space limitation.
\begin{figure}[t!]
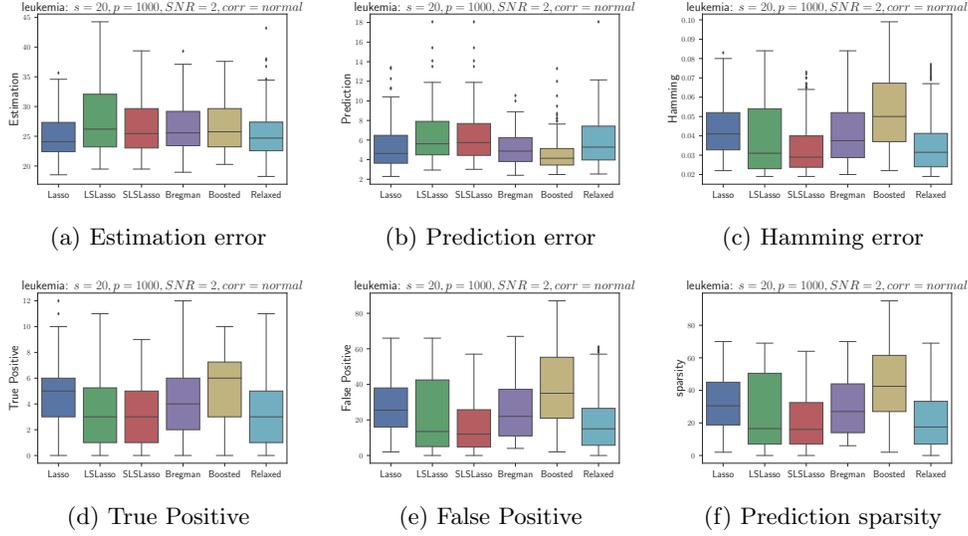

\centering
        \begin{subfigure}[b]{0.32\textwidth}
            \centering
            \includegraphics[width=0.99\textwidth]{{{code/graphs/s:20p:1000SNR:2-normal/Estimation}}}
            \caption{Estimation error}
        \end{subfigure}
        \begin{subfigure}[b]{0.32\textwidth}
            \centering
            \includegraphics[width=0.99\textwidth]{{{code/graphs/s:20p:1000SNR:2-normal/Prediction}}}
            \caption{Prediction error}
        \end{subfigure}
        \begin{subfigure}[b]{0.32\textwidth}
            \centering
            \includegraphics[width=0.99\textwidth]{{{code/graphs/s:20p:1000SNR:2-normal/Hamming}}}
            \caption{Hamming error}
        \end{subfigure}
        \begin{subfigure}[b]{0.32\textwidth}
            \centering
            \includegraphics[width=0.99\textwidth]{{{"code/graphs/s:20p:1000SNR:2-normal/True_Positive"}}}
            \caption{True Positive}
        \end{subfigure}
        \begin{subfigure}[b]{0.32\textwidth}
            \centering
            \includegraphics[width=0.99\textwidth]{{{"code/graphs/s:20p:1000SNR:2-normal/False_Positive"}}}
            \caption{False Positive}
        \end{subfigure}
        \begin{subfigure}[b]{0.32\textwidth}
            \centering
            \includegraphics[width=0.99\textwidth]{{{code/graphs/s:20p:1000SNR:2-normal/sparsity}}}
            \caption{Prediction sparsity}
        \end{subfigure}
        \caption{\emph{Leukemia} dataset with $p=1000$, $\SNR=2$, $s=20$, normal correlations scenario.}
        \label{plot:_leuk_1000_2_20}
\end{figure}

On~\Cref{plot:leuk_200_8_5} we notice that in the low noise ($\SNR = 8$) and very sparse ($s = 5$) scenario the overall performance is similar to the synthetic dataset, discussed above.

Meanwhile, the conclusion is different for~\Cref{plot:_leuk_1000_2_20}, where the noise level ($\SNR = 2$) and the sparsity level ($s = 20$) are high.
In this scenario we observe that simple Lasso outperforms in average all the refitting strategies in terms of estimation error and the TP rate, \Boosted provides with better prediction rate and the \SLSLasso and the {Relaxed Lasso} achieves better results in terms of all the other measures.
We additionally emphasize, that the \LSLasso shows large variance and might fail to improve the estimation in some cases.

Similar conclusions as for previous scenario can be made for~\Cref{plot:leuk_1000_8_20}, where the sparsity level $(s = 20)$ is still high, but the noise level is reduced $(\SNR = 8)$.
\begin{figure}[t!]
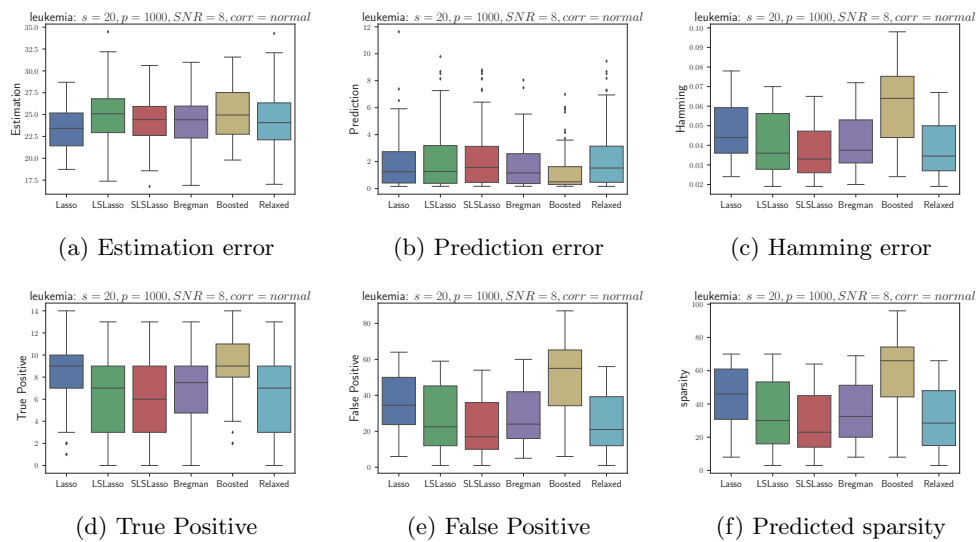

\centering
        \begin{subfigure}[b]{0.32\textwidth}
            \centering
            \includegraphics[width=0.99\textwidth]{{{code/graphs/s:20p:1000SNR:8-normal/Estimation}}}
            \caption{Estimation error}
        \end{subfigure}
        \begin{subfigure}[b]{0.32\textwidth}
            \centering
            \includegraphics[width=0.99\textwidth]{{{code/graphs/s:20p:1000SNR:8-normal/Prediction}}}
            \caption{Prediction error}
        \end{subfigure}
        \begin{subfigure}[b]{0.32\textwidth}
            \centering
            \includegraphics[width=0.99\textwidth]{{{code/graphs/s:20p:1000SNR:8-normal/Hamming}}}
            \caption{Hamming error}
        \end{subfigure}
        \begin{subfigure}[b]{0.32\textwidth}
            \centering
            \includegraphics[width=0.99\textwidth]{{{"code/graphs/s:20p:1000SNR:8-normal/True_Positive"}}}
            \caption{True Positive}
        \end{subfigure}
        \begin{subfigure}[b]{0.32\textwidth}
            \centering
            \includegraphics[width=0.99\textwidth]{{{"code/graphs/s:20p:1000SNR:8-normal/False_Positive"}}}
            \caption{False Positive}
        \end{subfigure}
        \begin{subfigure}[b]{0.32\textwidth}
            \centering
            \includegraphics[width=0.99\textwidth]{{{code/graphs/s:20p:1000SNR:8-normal/sparsity}}}
            \caption{Predicted sparsity}
        \end{subfigure}
        \caption{\emph{Leukemia} dataset with $p=1000$, $\SNR=8$, $s=20$, normal correlations scenario.}
        \label{plot:leuk_1000_8_20}
\end{figure}

We conclude by pointing out that performing or not performing the refitting depends on the measure of interest and on the underlying (unknown in practice) scenario.
\Boosted agrees with our theoretical results, as it only improves the Lasso estimator in terms of prediction, but it fails to improve any other measure except True Positive rate (due to higher output sparsity).
Least-Squares Lasso may show an undesirable performance in some scenarios and bring the problem of interpretability.
\SLSLasso and \Bregman showed consistent and satisfying performance, however Sign-Least-Squares Lasso outperforms both the \Bregman and the Least-Squares Lasso
in average.
We additionally emphasize that our results are valid for cross-validation on MSE.
Other choices are possible (BIC, AIC, AV$_{p}$) and may possibly provide with different overall conclusion.

\section*{Conclusion}

In this article we introduced a simple framework to provide additional statistical guarantees on the refitting and sign-consistent refitting of Lasso solutions.
We demonstrated that every sign-consistent refitting strategy satisfies an oracle inequality under the same assumptions as the Lasso bounds.
We theoretically analyzed two refitting strategies: \Boosted and \Bregman, which are easy to implement as they require only Lasso solver.
It appeared that the \Bregman converges to the \SLSLasso, a particular refitting strategy with sign preserving properties.
Experimental results show the advantages of sign-consistent strategies over the simple \LSLasso.
Possible extension of this work is to consider other families of the refitting strategies by either taking into account an additional information provided by the Lasso or replacing the sign-preserving property.
Another interesting road is to use our framework to provide oracle inequalities for estimation error and feature selection.

\section*{Acknowledgment}

This work was partially supported by "Laboratoire d'excellence B\'ezout of Universit\'e Paris Est" and by  "Chair Machine Learning for Big Data at T\'el\'ecom ParisTech".

\vskip 0.2in
\bibliographystyle{plainnat}
\bibliography{final_refs}
\newpage

\appendix

\section{Proofs}
\label{sec:appendix_proofs}

\begin{proof}[Proof of~\Cref{thm:oracle_lasso}]
    We start from the KKT conditions for Lasso~\Cref{lem:kkt_lasso},
    noticing that $\scalar{\hrho}{\hbeta} = \normin{\hbeta}_1$ and for every $\beta \in \bbR^p$ we have $\scalar{\hrho}{\beta} \leq \normin{\beta}_1$ since $\normin{\hrho}_{\infty}\leq 1$, we can write
    \begin{align*}
        \frac{1}{n}(\beta - \hbeta)^\top X^\top (y - X\hbeta) = \scalar{\beta - \hbeta}{\hrho} & \leq {\lambda}(\normin{\beta}_1 - \normin{\hbeta}_1)\enspace,
    \end{align*}
    hence, we have $\frac{1}{n}(\beta - \hbeta)^\top X^\top X(\tbeta - \hbeta) \leq {\lambda}(\normin{\beta}_1 - \normin{\hbeta}_1) + \frac{1}{n} X^\top \varepsilon (\hbeta - \beta)$.

    Using $2 \scalar{x}{y}  = \norm{x}_2^2 + \norm{y}_2^2 -\norm{x-y}_2^2$ (for all $x,\, y \in \bbR^p$),
    we derive
    \begin{align*}
            \tfrac{1}{2n}\normin{X(\beta - \hbeta)}_2^2 + \tfrac{1}{2n}\normin{X(\beta^* - \hbeta)}_2^2 -\tfrac{1}{2n}\normin{X(\beta - \beta^*)}_2^2 \leq &{\lambda}(\normin{\beta}_1 - \normin{\hbeta}_1) + \tfrac{1}{n} X^\top \varepsilon (\hbeta - \beta)\enspace.
    \end{align*}
    From~\Cref{lem:union_bound}, with probability at least $1-\delta$, $\norm{X^\top \varepsilon / n}_{\infty} \leq {\lambda} / 2$. This yields
    \begin{align}
        \label{Basic_ineq}
        \tfrac{1}{n}\normin{X(\beta - \hbeta)}_2^2 + \tfrac{1}{n}\normin{X(\hbeta - \beta^*)}_2^2 \leq \tfrac{1}{n}\normin{X(\beta - \beta^*)}_2^2 + {\lambda}(\normin{\hbeta - \beta}_1 + 2\normin{\beta}_1 - 2\normin{\hbeta}_1)\enspace,
    \end{align}
    now let us set $\hat\Delta = \hbeta - \beta$ and let $J \in [p]$ be any set such that $|J| \leq s$. Using the triangle inequality and the decomposability of the $\ell_1$-norm, we can write
    \begin{align*}
        \normin{\hbeta - \beta}_1 &+ 2\normin{\beta}_1 - 2\normin{\hbeta}_1 \leq 3\normin{\hat\Delta_J}_1 - \normin{\hat\Delta_{J^c}}_1 + 4\normin{\beta_{J^c}}_1\enspace.
    \end{align*}
    We consider two cases:
    \begin{itemize}
        \item if $3\normin{\hat\Delta_J}_1 - \normin{\hat\Delta_{J^c}} \leq 0$, then~\cref{Basic_ineq} implies
        \begin{equation*}
            \frac{1}{n}\normin{X(\beta - \hbeta)}_2^2 + \frac{1}{n}\normin{X(\hbeta - \beta^*)}_2^2 \leq \frac{1}{n}\normin{X(\beta - \beta^*)}_2^2 + 4\lambda\normin{\beta_{J^c}}_1\enspace,
        \end{equation*}
        \item if $3\normin{\Delta_J}_1 - \normin{\hat\Delta_{J^c}} > 0$, then RE($3$, $s$) implies
         $\normin{\hat\Delta_J}_2^2 \leq \frac{\normin{X\hat\Delta}_2^2}{n\kappa^2(3, s)},$
        and hence in view of~\cref{Basic_ineq} we get
        \begin{align*}
            \frac{1}{n}\normin{X(\hbeta - \beta^*)}_2^2 &\leq \frac{1}{n}\normin{X(\beta - \beta^*)}_2^2 + 4\lambda\normin{\beta_{J^c}}_1 + 3\lambda\normin{\hat\Delta_J}_1 - \frac{1}{n}\normin{X\hat\Delta}_2^2 \\
            &\leq \frac{1}{n}\normin{X(\beta - \beta^*)}_2^2 + 4\lambda\normin{\beta_{J^c}}_1 + 3{\lambda}\sqrt{s}\normin{\hat\Delta_J}_2 - \frac{1}{n}\normin{X\hat\Delta}_2^2\\
            &\leq \frac{1}{n}\normin{X(\beta - \beta^*)}_2^2 + 4\lambda\normin{\beta_{J^c}}_1 + \frac{3{\lambda}\sqrt{s}}{\kappa(3, s)}\frac{1}{\sqrt{n}}\normin{X\hat\Delta}_2 - \frac{1}{n}\normin{X\hat\Delta}_2^2 \enspace.
        \end{align*}
        If we set $x = \normin{X\hat\Delta}_2/\sqrt{n}$, the last two terms are
        \begin{equation*}
            \sqrt{s}\frac{3{\lambda}}{\kappa(3, s)}x - x^2 = \Big(\sqrt{s}\frac{3{\lambda}}{2\kappa(3, s)}\Big)^2 - \Big(\sqrt{s}\frac{3{\lambda}}{2\kappa(3, s)} - x\Big)^2 \leq \frac{9{\lambda}s}{4\kappa^2(3, s)}\enspace.
        \end{equation*}
    \end{itemize}
    and we get the desired result
    \begin{align*}
        \frac{1}{n}\normin{X(\beta^* - \hbeta)}_2^2 \leq \inf_{\beta \in \bbR^p} \min_{|J| \leq s}\Big\{\frac{1}{n}\norm{X(\beta - \beta^*)}_2^2 + \frac{9\lambda^2 s}{4 \kappa^2(3, s)} + 4{\lambda}\normin{\beta_{J^c}}_1 \Big\}\enspace.
    \end{align*}
\end{proof}

\begin{proof}[Proof of~\Cref{thm:oracle_refitting_lasso}]
    Substituting $y = X\beta^* + \varepsilon$ into~\Cref{def:refititng_vector} to get
    \begin{align*}
        \frac{1}{2n}\norm{X(\bbeta - \beta^*)}_2^2 - \frac{1}{2n}\norm{X(\beta^* - \hbeta)}_2^2 &\leq \frac{1}{n} X^\top \varepsilon(\bbeta - \hbeta)\enspace.
    \end{align*}
    Hence, we can conclude using H\"older's inequality on the event  $\{\normin{ X^\top \varepsilon/n}_{\infty} \leq \lambda/2\}$.
\end{proof}

\begin{proof}[Proof of~\Cref{thm:oracle_sign_refitting_lasso}]
    Let $\bbeta$ be a sign-consistent refitting of $\hat\beta = \hbeta(\lambda)$ - Lasso solution with the tuning parameter $\lambda$, we first define
    \begin{align}\label{eq:def_deltas}
        \Delta = \bbeta - \hat\beta,\quad
        \bar\Delta = \tbeta - \bbeta,\quad
        \hat\Delta = \tbeta - \hbeta\enspace.
    \end{align}
    We start with the KKT conditions for Lasso (\Cref{lem:kkt_lasso}).
    Since the refitting is sign-consistent we have $\scalar{\bbeta}{\hrho} = \normin{\bbeta}_1$, therefore one can write
    \begin{align}
        \frac{1}{n}(\tbeta - \hbeta)^\top X^\top(y - X\hbeta) &= \lambda\scalar{\tbeta - \hbeta}{\hrho} \leq \lambda(\normin{\tbeta}_1 - \normin{\hbeta}_1)\enspace,\label{eq:subgrad_l1_lasso}\\
        \frac{1}{n}(\tbeta - \bbeta)^\top X^\top(y - X\hbeta) &= \lambda\scalar{\tbeta - \bbeta}{\hrho} \leq \lambda(\normin{\tbeta}_1 - \normin{\bbeta}_1)\enspace.\label{eq:subgrad_l1_sign_refitting}
    \end{align}
    Substituting the model in~\Cref{eq:linear_regression_model}, we get
    \begin{align*}
        \frac{1}{n}(\tbeta - \hbeta)^\top X^\top X(\tbeta - \hbeta) &\leq \lambda(\normin{\tbeta}_1 - \normin{\hbeta}_1) + \frac{1}{n} X^\top \varepsilon(\hbeta - \tbeta)\enspace,\\
        \frac{1}{n}(\tbeta - \bbeta)^\top X^\top X(\tbeta - \hbeta) &\leq \lambda(\normin{\tbeta}_1 - \normin{\bbeta}_1)+ \frac{1}{n} X^\top \varepsilon(\bbeta - \tbeta)\enspace.
    \end{align*}
    Therefore, we have three ingredients to derive our final result: the first one, given in~\cref{eq:subgrad_l1_lasso}, relies on the subgradient property of the $\ell_1$-norm applied to the Lasso solution; the second one, given in~\cref{eq:subgrad_l1_sign_refitting}, relies on the same property applied to the sign-consistent refitting; and the third one
    is coming from the definition of a refitting strategy~\Cref{def:refititng_vector}.
    \begin{align*}
        \frac{1}{n}\norm{X\hat\Delta}_2^2&\leq {\lambda}(\normin{\tbeta}_1 - \normin{\hbeta}_1) + \frac{1}{n} X^\top \varepsilon (\hbeta - \tbeta)\enspace,\\
        \frac{1}{2n}\norm{X\hat\Delta}_2^2 + \frac{1}{2n}\norm{X\bar\Delta}_2^2  -\frac{1}{2n}\norm{X\Delta}_2^2 &\leq {\lambda}(\normin{\tbeta}_1 - \normin{\bbeta}_1) + \frac{1}{n} X^\top \varepsilon (\bbeta - \tbeta)\enspace,\\
        \frac{1}{2n}\norm{X\bar\Delta}_2^2 - \frac{1}{2n}\norm{X\hat\Delta}_2^2 &\leq \frac{1}{n} X^\top \varepsilon(\bbeta - \hbeta)\enspace.
    \end{align*}
    Multiplying the second inequality by $1/2$ and summing the three equations we get
    \begin{align}
    \label{eq:mainRefitt}
        \frac{3}{4n}\norm{X\bar\Delta}_2^2 + \frac{3}{4n}\norm{X\hat\Delta}_2^2 \leq &\frac{1}{4n}\norm{X\Delta}_2^2 + \lambda\Big\{\frac{3}{2}\normin{\tbeta}_1 - \frac{1}{2}\normin{\bbeta}_1 - \normin{\hbeta}_1\Big\}\nonumber\\
        &+ \frac{3}{2n} X^\top \varepsilon(\bbeta - \tbeta)\enspace.
    \end{align}
    Now observe that
    \begin{align*}
    \normin{X\Delta}_2^2 - \normin{X\hat\Delta}_2^2 & = \normin{X\bar\Delta}_2^2 - 2\bar\Delta^\top X^\top X\hat\Delta \\  &  \leq  \normin{X\bar\Delta}_2^2 + 2\normin{X\bar\Delta}_2\normin{X\hat\Delta}_2\\ & \leq \normin{X\bar\Delta}_2^2 + \normin{X\bar\Delta}_2^2 + \normin{X\hat\Delta}_2^2\enspace,
    \end{align*}
    where we have again used $2ab \leq a^2 + b^2$ in the last inequality. Then, subtracting $\frac{1}{4n}\normin{X\hat\Delta}_2^2$ from the both sides in~\Cref{eq:mainRefitt} and using previous inequality, we get
    \begin{align*}
        \frac{1}{n}\normin{X\bar\Delta}_2^2 + \frac{1}{n}\normin{X\hat\Delta}_2^2
        \leq&\lambda\Big\{6\normin{\tbeta}_1 - 2\normin{\bbeta}_1 - 4\normin{\hbeta}_1\Big\} + \frac{6}{n} X^\top \varepsilon(\bbeta - \tbeta)\\
        \leq& \lambda\Big\{6\normin{\tbeta_S}_1 - 2\normin{\tbeta_S}_1 + 2\normin{\bar\Delta_S}_1 - 2\normin{\bar\Delta_{S^c}}_1 \\
        &- 4\normin{\tbeta_S}_1 + 4\normin{\hat\Delta_S}_1 - 4\normin{\hat\Delta_{S^c}}_1\Big\} + \frac{6}{n} X^\top \varepsilon(\bbeta - \tbeta)\\
        =& \lambda\Big\{2\normin{\bar\Delta_S}_1 - 2\normin{\bar\Delta_{S^c}}_1\Big\} + \lambda\Big\{4\normin{\hat\Delta_S}_1 - 4\normin{\hat\Delta_{S^c}}_1\Big\} \\
        &+ \frac{6}{n} X^\top \varepsilon(\bbeta - \tbeta)\enspace,
    \end{align*}
    where in the last inequality we used the fact that $\normin{\beta}_1 \geq \normin{\tbeta_S}_1 - \normin{\tbeta_S - \beta_S}_1 + \normin{\beta_{S^c}}$ for any vector $\beta$ and $\tbeta$ with $\supp(\tbeta) = S$. Now let us restrict our attention on the event where $\normin{ X^\top \varepsilon/n}_{\infty} \leq \lambda/6$, hence we can write
    \begin{align*}
        \frac{1}{n}\normin{X\bar\Delta}_2^2 + \frac{1}{n}\normin{X\hat\Delta}_2^2
        &\leq \lambda\{3\normin{\bar\Delta_S}_1 - \normin{\bar\Delta_{S^c}}_1\} + \lambda\{4\normin{\hat\Delta_S}_1 - 4\normin{\hat\Delta_{S^c}}_1\}\enspace,
    \end{align*}
    notice that both $4\normin{\hat\Delta_S}_1 - 4\normin{\hat\Delta_{S^c}}_1$ and $3\normin{\bar\Delta_S}_1 - \normin{\bar\Delta_{S^c}}_1$ can not be negative simultaneously, if one of them is negative we can simply erase it. W.l.o.g. we can assume that both terms are positive, hence we have
    \begin{align*}
        3\normin{\bar\Delta_S}_1 &\geq \normin{\bar\Delta_{S^c}}_1\\
        \normin{\hat\Delta_S}_1 &\geq \normin{\hat\Delta_{S^c}}_1\enspace,
    \end{align*}
    which allows us to write
    \begin{align*}
        \frac{1}{n}\normin{X\bar\Delta}_2^2 + \frac{1}{n}\normin{X\hat\Delta}_2^2
        &\leq \lambda\{3\normin{\bar\Delta_S}_1 - \normin{\bar\Delta_{S^c}}_1\} + \lambda\{4\normin{\hat\Delta_S}_1 - 4\normin{\hat\Delta_{S^c}}_1\}\\
        &\leq 3\lambda\normin{\bar\Delta_S}_1 + 4\lambda\normin{\hat\Delta_S}_1
        \leq 3\lambda\sqrt{s}\normin{\bar\Delta_S}_2 + 4\lambda\sqrt{s}\normin{\hat\Delta_S}_2\\
        &\leq \frac{3\lambda\sqrt{s}}{\kappa(3, s)\sqrt{n}}\normin{X\bar\Delta}_2 + \frac{4\lambda\sqrt{s}}{\kappa(1, s)\sqrt{n}}\normin{X\hat\Delta}_2\\
        &\leq \frac{9\lambda^2s}{2\kappa^2(3, s)} + \frac{1}{2n}\normin{X\bar\Delta}_2^2 + \frac{16\lambda^2 s}{2\kappa^2(1, s)} + \frac{1}{2n}\normin{X\hat\Delta}_2^2\enspace,
    \end{align*}
    again in the last inequality we used $2ab \leq a^2 + b^2$. Hence, we can write
    \begin{align}
        \frac{1}{2n}\normin{X(\tbeta - \bbeta)}_2^2 + \frac{1}{2n}\normin{X(\tbeta - \hbeta)}_2^2 \leq \lambda^2s\Big(\frac{9}{2\kappa^2(3, s)} + \frac{8}{\kappa^2(1, s)}\Big)\enspace.
    \end{align}
\end{proof}

\begin{proof}[Proof of~\Cref{prop:GoodProp}]
    For the second part we notice that $\hbeta^{\lambda_1} = 0$ (since $\lambda_1 \geq \lambda_{1, \max}$) and the statement follows from~\Cref{eq:lasso_boosting} with $\hbeta^{\lambda_1} = 0$.
    For the first part we can compute the critical value for~\Cref{eq:lasso_boosting} (it is a Lasso problem), which is given by $\lambda_{2, \max} = \normin{X^\top \bar y / n}_{\infty}$, where $\bar y = y - X\hbeta^{\lambda_1}$.
    Notice that thanks to~\Cref{lem:kkt_lasso} we can write
    \[
        \lambda_{2, \max} = \frac{1}{n}\big\|X^\top \bar y\big\|_{\infty} = \frac{1}{n}\norm{X^\top \big(y - X\hbeta^{\lambda_1}\big)}_{\infty} = \lambda_1\norm{\hrho^{\lambda_1}}_{\infty}, \,\,\text{for some } \hrho^{\lambda_1} \in \sign(\hbeta^{\lambda_1})\enspace,
    \]
    since $\lambda_1 < \lambda_{1, \max}$, hence $\hbeta^{\lambda_1} \neq 0$ and $\normin{\hrho^{\lambda_1}}_{\infty} = 1$, which concludes the proof.
\end{proof}

\begin{proof}[Proof of~\Cref{lem:boosting_lemma}]
By optimality of $\hbeta^{\lambda_1, \lambda_2}$ in~\cref{eq:lasso_boosting_strange_form}
    \begin{align*}
        \frac{1}{2n}\norm{y - X\hbeta^{\lambda_1, \lambda_2}}_2^2 + \lambda_2\normin{\hbeta^{\lambda_1, \lambda_2} - \hbeta^{\lambda_1}}_1 &\leq\frac{1}{2n}\norm{y - X\hbeta^{\lambda_1}}_2^2\enspace,
    \end{align*}
    hence on the event $\normin{ X^\top \varepsilon/n}_{\infty} \leq {\lambda_1}/{2}$ and using H\"older's inequality, we can write
    \begin{align*}
        \tfrac{1}{2n}\norm{X(\tbeta - \hbeta^{\lambda_1, \lambda_2})}_2^2 + \lambda_2\normin{\hbeta^{\lambda_1, \lambda_2} - \hbeta^{\lambda_1}}_1 &\leq \tfrac{1}{2n}\norm{X(\tbeta - \hbeta^{\lambda_1})}_2^2 + \tfrac{1}{n} X^\top \varepsilon(\hbeta^{\lambda_1, \lambda_2} - \hbeta^{\lambda_1})\\
        &\leq \tfrac{1}{2n}\norm{X(\tbeta - \hbeta^{\lambda_1})}_2^2  + \tfrac{\lambda_1}{2}\normin{\hbeta^{\lambda_1, \lambda_2} - \hbeta^{\lambda_1}}_1\enspace.
    \end{align*}
\end{proof}

\begin{proof}[Proof of~\Cref{cor:boosted_oracle}]
Combining~\Cref{lem:boosting_lemma} and~\Cref{thm:oracle_lasso} we can state the following corollary
\end{proof}

\begin{proof}[Proof of~\Cref{thm:bregman_racle}]
    Let $r = \lambda_2 / \lambda_1$ and $\lambda = \lambda_1$, from the optimality of the Lasso and the Bregman Lasso we have
    \begin{align*}
      \frac{1}{n}X^\top X(\tbeta - \hbeta) + \frac{1}{n} X^\top \varepsilon &= \lambda \hrho\enspace,\\
      \frac{1}{n}X^\top X(\tbeta - \bbeta) + \frac{1}{n} X^\top \varepsilon &= r\lambda(\brho - \hrho)\enspace,
    \end{align*}
    where $\brho$ and $\hrho$ are subgradients of $\bbeta$ (Bregman Lasso) and $\hbeta$ (Lasso) respectively.
    Hence, we can write as
    \begin{align}
      \label{eq:optimility_;asso}
      \frac{1}{n}X^\top X(\tbeta - \hbeta) + \frac{1}{n} X^\top \varepsilon &= \lambda \hrho\enspace,\\
      \label{eq:optimality_bregman}
      \frac{1}{r n}X^\top X(\tbeta - \bbeta) + \frac{1}{n}X^\top X(\tbeta - \hbeta) + \frac{1+r}{r}\frac{1}{n} X^\top \varepsilon&= \lambda\brho\enspace.
    \end{align}
    Multiplying the first equality by $\bbeta - \hbeta$, the second by $\hbeta - \bbeta$, and using the notation $\hat{\Sigma} = X^\top X /n$ we can write:
    \begin{align*}
      &(\bbeta - \hbeta)^\top \hat{\Sigma}(\tbeta - \hbeta) + \frac{1}{n}(\bbeta - \hbeta)^\top X^\top \varepsilon &\leq \lambda (\normin{\bbeta}_1 - \normin{\hbeta}_1)\enspace,\\
      \frac{1}{r}(\hbeta - \bbeta)^\top \hat{\Sigma}(\tbeta - \bbeta)
      + &(\hbeta - \bbeta)^\top \hat{\Sigma}(\tbeta - \hbeta) + \frac{1+r}{n r} (\hbeta - \bbeta)^\top X^\top \varepsilon&\leq \lambda(\normin{\hbeta}_1 - \normin{\bbeta}_1)\enspace.
    \end{align*}
    Summing up the previous inequalities yields:
    \begin{align*}
      \frac{1}{r}(\hbeta - \bbeta)^\top \hat{\Sigma}(\tbeta - \bbeta) + \frac{1}{rn} (\hbeta - \bbeta)^\top X^\top \varepsilon \leq 0\enspace.
    \end{align*}
    We recall the following notation:
    $\Delta = \bbeta - \hat\beta,\quad
     \bar\Delta = \tbeta - \bbeta,\quad
     \hat\Delta = \tbeta - \hbeta$.
    Hence, the previous inequality can be simplified as
    \begin{align*}
        \frac{1}{2n}\normin{X\bar\Delta}_2^2 + \frac{1}{2n}\normin{X\Delta}_2^2 - \frac{1}{2n}\normin{X\hat\Delta}_2^2 \leq \frac{1}{n} (\bbeta - \hbeta)^\top X^\top \varepsilon\enspace.
    \end{align*}
    Additionally notice, that the lasso itself admits the following bound:
    \begin{align*}
      \frac{1}{n}\normin{X\hat{\Delta}}_2^2 \leq \lambda(\normin{\tbeta}_1 - \normin{\hbeta}_1) + \frac{1}{n} (\hbeta - \tbeta)^\top X^\top \varepsilon\enspace,
    \end{align*}
    which, combined with the previous inequality gives:
    \begin{align}
        \label{eq:bregman_proof_lasso_oracle_bregman_mix}
        \frac{1}{2n}\normin{X\bar\Delta}_2^2 + \frac{1}{2n}\normin{X\Delta}_2^2 + \frac{1}{2n}\normin{X\hat\Delta}_2^2 \leq \lambda(\normin{\tbeta}_1 - \normin{\hbeta}_1) + \frac{1}{n} (\bbeta - \tbeta)^\top X^\top \varepsilon \enspace.
    \end{align}
    Multiplying~\cref{eq:optimality_bregman} by $(\tbeta - \bbeta)$ from both sides we arrive at:
    \begin{align*}
      (2+r)\frac{1}{2n}\normin{X\bar\Delta}_2^2 + r\frac{1}{2n}\norm{X\hat\Delta}_2^2  -r\frac{1}{2n}\norm{X\Delta}_2^2 \leq &r\lambda(\normin{\tbeta}_1 - \normin{\bbeta}_1)\\
      &+ (1+r)\frac{1}{n} X^\top \varepsilon(\bbeta - \tbeta)\enspace.
    \end{align*}
    Summing the previous inequality with~\cref{eq:bregman_proof_lasso_oracle_bregman_mix} multiplied by $r$ we arrive at
    \begin{align*}
      \frac{r + 1}{r}\frac{1}{n}\normin{X\bar\Delta}_2^2 + \frac{1}{n}\norm{X\hat\Delta}_2^2 \leq &\lambda(2\normin{\tbeta}_1 - \normin{\bbeta}_1 - \normin{\hbeta}_1)\\
      &+ \frac{1 + 2r}{r}\frac{1}{n} X^\top \varepsilon(\bbeta - \tbeta)\enspace.
    \end{align*}
    We continue on the event $\normin{({1 + 2r}) X^\top \varepsilon/ rn}_{\infty} \leq \lambda / 2$ as:
    \begin{align*}
      \frac{r + 1}{r}\frac{1}{n}\normin{X\bar\Delta}_2^2 + \frac{1}{n}\norm{X\hat\Delta}_2^2 \leq &\lambda(2\normin{\tbeta}_1 - \normin{\bbeta}_1 - \normin{\hbeta}_1) + \frac{1 + 2r}{r}\frac{1}{n} X^\top \varepsilon(\bbeta - \tbeta)\\
      = & \lambda(\normin{\hat\Delta_S}_1 - \normin{\hat\Delta_{S^c}}_1 + \frac{3}{2}\normin{\bar\Delta_S}_1 - \frac{1}{2}\normin{\bar\Delta_{S^c}}_1)\enspace.
    \end{align*}
    Hence, we have
    \begin{align*}
      2\frac{r + 1}{r}\frac{1}{n}\normin{X\bar\Delta}_2^2 + \frac{2}{n}\norm{X\hat\Delta}_2^2 \leq & \lambda(2\normin{\hat\Delta_S}_1 - 2\normin{\hat\Delta_{S^c}}_1 + 3\normin{\bar\Delta_S}_1 - \normin{\bar\Delta_{S^c}}_1)\enspace.
    \end{align*}
    By similar arguments as in~\Cref{thm:oracle_sign_refitting_lasso} we can use the Restricted Eigenvalue assumption and derive the following sequence of inequalities:
    \begin{align*}
      2\frac{r + 1}{r}\frac{1}{n}\normin{X\bar\Delta}_2^2 + \frac{2}{n}\norm{X\hat\Delta}_2^2 \leq &\lambda(2\normin{\hat\Delta_S}_1 - 2\normin{\hat\Delta_{S^c}}_1 + 3\normin{\bar\Delta_S}_1 - \normin{\bar\Delta_{S^c}}_1)\\
      \leq &2\sqrt{\frac{s}{n}}\frac{\lambda}{\kappa(1, s)}\normin{X\hat\Delta}_2 + 3\sqrt{\frac{s}{n}}\frac{\lambda}{\kappa(3, s)}\normin{X\bar\Delta}_2\\
      \leq &\frac{\lambda^2 s}{\kappa^2(1, s)}+\frac{1}{n}\normin{X\hat\Delta}_2^2 + \frac{9\lambda^2 s}{4\kappa^2(3, s)} + \frac{1}{n}\normin{X\bar\Delta}_2^2\enspace.
    \end{align*}
    Finally,
    \begin{align*}
      \frac{r + 2}{r}\frac{1}{n}\normin{X\bar\Delta}_2^2 + \frac{1}{n}\norm{X\hat\Delta}_2^2
      \leq &\frac{\lambda^2 s}{\kappa^2(1, s)} + \frac{9\lambda^2 s}{4\kappa^2(3, s)}\enspace.
    \end{align*}
\end{proof}

\begin{proof}[Proof of~\Cref{lemma:props_bregman_divergence}]
    The identity $\Breg{\ell_1}^{\rho}({z}, w) = \norm{{z}}_1 - \scalar{{\rho}}{{z}}$ immediately follows from the definition of the Bregman divergence and the fact that $\scalar{\rho}{w} = \normin{w}_1$.
    To prove convexity, we consider $z_1, z_2, w \in \bbR^p$ arbitrary vectors and $t \in [0, 1]$, hence by the definition of the Bregman divergence we can write
    \begin{align*}
            \Breg{\ell_1}^{\rho}(tz_1 + (1-t)z_2, w) = &\normin{tz_1 + (1-t)z_2}_1 - \normin{w}_1 - \scalar{\rho}{tz_1 + (1-t)z_2 - w} \\
            \leq &t(\norm{z_1}_1 - \scalar{\rho}{z_1}) + (t-1)(\norm{z_2}_1 -\scalar{\rho}{z_2}) +  \scalar{\rho}{w} - \normin{w}_1\\
            = &t\Breg{\ell_1}^{\rho}(z_1, w) + (1-t)\Breg{\ell_1}^{\rho}(z_2, w)\enspace,
    \end{align*}
    where we used the triangle inequality and the identity $\Breg{\ell_1}^{\rho}({z}, w) = \norm{{z}}_1 - \scalar{{\rho}}{{z}}$.
    The bound $0 \leq \Breg{\ell_1}^{\rho}({z}, w) \leq 2\norm{{z}}_1$, follows from H\"older's inequality and the fact that $\normin{\rho}_{\infty} \leq 1$.
    Separability $\Breg{\ell_1}^{\rho}({z}, w) = \sum_{i = 1}^p |z_i| - z_i\rho_i  = \sum_{i = 1}^p \Breg{\ell_1}^{\rho_i}(z_i, w_i)$ is a consequence of the separability of the $\ell_1$-norm.
    Finally, the last property follows from the definition of the subgradient $\rho$ and identity $\Breg{\ell_1}^{\rho}({z}, w) = \norm{{z}}_1 - \scalar{{\rho}}{{z}}$.
\end{proof}

\begin{proof}[Proof of~\Cref{prop:bregman_as_Lasso}]
    One can write the following sequence of equalities
    \begin{align*}
        \frac{1}{2n}\norm{y - X\beta}_2^2 + \lambda_2\Breg{\ell_1}^{\hrho^{\lambda_1}}(\beta, \hbeta^{\lambda_1}) &=  \frac{1}{2n}\norm{y}_2^2 - \frac{1}{n}\scalar{X^\top y}{\beta} + \frac{1}{2n}\norm{X\beta}_2^2 + \lambda_2\Breg{\ell_1}^{\hrho^{\lambda_1}}(\beta, \hbeta^{\lambda_1}) \\
        &=\frac{1}{2n}\norm{y}_2^2 - \frac{1}{n}\scalar{X^\top \bar y}{\beta} + \frac{1}{2n}\norm{X\beta}_2^2 + \lambda_2\normin{\beta}_1\enspace,
    \end{align*}
    where to get the last equality we used Property~\ref{prop:breg_via_subgrad} and the choice of subgradient given in~\Cref{eq:refitting_subgradient}.
    Noticing the following relation
    \begin{align}
      \frac{1}{2n}\norm{y}_2^2 - \frac{1}{n}\scalar{X^\top \bar y}{\beta} + \frac{1}{2n}\norm{X\beta}_2^2  = \frac{1}{2n}\norm{\bar y - X\beta}_2^2
      + \frac{1}{2n}\norm{y}_2^2 - \frac{1}{2n}\norm{\bar y}_2^2\nonumber\enspace,
    \end{align}
    and using the fact that $y, \bar y$ are independent of $\beta$, we conclude.
\end{proof}

\begin{proof}[Proof of~\Cref{lem:bregman_is_refitting}]
    The proof of this lemma relies on~\Cref{eq:bregman_iteration_linear}, hence we can write
    \begin{align*}
        \frac{1}{2n}\normin{y - X\hbeta^{\lambda_1, \lambda_2}}_2^2 + \lambda_2\Breg{\ell_1}^{\hrho^{\lambda_1}}(\hbeta^{\lambda_1, \lambda_2}, \hbeta^{\lambda_1}) \leq \frac{1}{2n}\normin{y - X\hbeta^{\lambda_1}}_2^2\enspace,
    \end{align*}
    so that we get the result since Bregman divergences are non negative.
\end{proof}

\begin{proof}[Proof of~\Cref{prop:orthogonal_subgradien_same_sign}]
   Since $\hbeta^{\lambda_1} = \ST(y, \lambda_1)$, we can equivalently rewrite~\cref{eq:orthogonal_subgradient} as
    \begin{align*}
        \forall i \in [p], \quad \hrho^{\lambda_1}_i = \begin{cases}
                                                            \sign(y_i), \quad &|y_i| \geq \lambda_1\\
                                                            \frac{y_i}{\lambda_1},\quad &|y_i| < \lambda_1
                                                     \end{cases}\enspace.
    \end{align*}
\vspace{-0.4cm}
\end{proof}

\begin{proof}[Proof of~\Cref{prop:solution_of_refitting_with_orthogonal design}]
    We have the following sequence of equalities
    \begin{align*}
        \frac{1}{2}\norm{y - \beta}_2^2 + \lambda_2\Big(\norm{\beta}_1 - \scalar{\hrho^{\lambda_1}}{\beta}\Big) &= \frac{\norm{y}_2^2}{2} - \scalar{y + \lambda_2\hrho^{\lambda_1}}{\beta} + \frac{\norm{\beta}_2^2}{2} + \lambda_2\norm{\beta}_1\\
        &=\frac{1}{2}\norm{y + \lambda_2\hrho^{\lambda_1} - \beta}_2^2 + \lambda_2\norm{\beta}_1 + \frac{\norm{y}_2^2}{2} - \frac{\norm{y - \lambda_2\hrho^{\lambda_1}}_2^2}{2}\enspace.
    \end{align*}
    Since, $y$ and $\hrho^{\lambda_1}$ do not depend on $\beta$ we can write
    \begin{align}
        \label{eq:orthogonal_refitting_alternative}
        \hbeta^{\lambda_1, \lambda_2} = \argmin_{\beta \in \bbR^p} \frac{1}{2}\norm{y + \lambda_2\hrho^{\lambda_1} - \beta}_2^2 + \lambda_2\norm{\beta}_1\enspace,
    \end{align}
    which reduces us to the case of Lasso with orthogonal design and we get the desired result.
\end{proof}
\begin{proof}[Proof of~\Cref{prop:orthogonal_refitting_is_mcp}]
        First we notice that the optimization problem in~\Cref{eq:orthogonal_refitting} is separable \ie can be solved component-wise and that $\mu\gamma = \lambda_H(1 + \frac{\lambda_1}{\lambda_2}) = \lambda_1$ and $\frac{\gamma}{\gamma - 1}
        = 1 + \frac{\lambda_2}{\lambda_1}$. For each $j \in [p]$ the solution of the refitting step~\cref{eq:bregman_iteration_linear} is given by
    \begin{equation*}
        \hbeta^{\lambda_1, \lambda_2}_j = \sign(y_j)\big(|y_j + \frac{\lambda_2}{\lambda_1}(y_j - \hbeta^{\lambda_1}_j)| - \lambda_2\big)_+\enspace,
    \end{equation*}
    where we used the fact that $\sign(y_j) = \sign(y_j + \lambda_2\hrho^{\lambda_1}_j)$, a result that follows from~\Cref{prop:orthogonal_subgradien_same_sign}. We consider the following cases:
    \begin{itemize}
        \item $y_j > 0, y_j \geq \lambda_1$, hence $\hbeta^{\lambda_1}_j = y_j - \lambda_1$ and
            \begin{equation*}
                \sign(y_j)\big(|y_j + \frac{\lambda_2}{\lambda_1}(y_j - \hbeta^{\lambda_1}_j)| - \lambda_2\big)_+ = \big(y_j + \frac{\lambda_2}{\lambda_1}(y_j - y_j + \lambda_1) - \lambda_2\big)_+ = y_j\enspace,
            \end{equation*}
        which holds for all positive values of $\lambda_2$.
        \item $y_j > 0, y_j < \lambda_1$, hence $\hbeta^{\lambda_1}_j = 0$ and
            \begin{align*}
                \sign(y_j)\big(|y_j + \frac{\lambda_2}{\lambda_1}(y_j - \hbeta^{\lambda_1}_j)| - \lambda_2\big)_+ &= \big(y_j + \frac{\lambda_2}{\lambda_1}y_j - \lambda_2\big)_+
                = \big(y_j\big(1 + \frac{\lambda_2}{\lambda_1}\big) - \lambda_2\big)_+\\
                &= \big(1 + \frac{\lambda_2}{\lambda_1}\big)\big(y_j - \lambda_H\big)_+\enspace.
            \end{align*}
    \end{itemize}
    The case of negative $y_j$ is proved in the same manner.
\end{proof}
\begin{proof}[Proof of~\Cref{prop:breg_isfirw}]
    We prove this result component-wise (for arbitrary component $j \in [p]$) and additionally assume that $y_j > 0$, the case of negative $y_j$ being similar.
    The statement clearly holds for $k = 2$, since it is sufficient to use \Cref{prop:orthogonal_refitting_is_mcp} with $\lambda_1 = \lambda_2$.
    We assume that the statement holds up to $k > 2$, hence using previous result we can write
    \begin{equation*}
      \hbeta_{k+1, j} = \ST(y_j + \lambda\hrho_{k,j}, \lambda)\enspace,
    \end{equation*}
    reminding $\rho_0=0$, with the inductive assumption and the definition of the subgradient in~\cref{eq:old_bregman_iteration} we have
    \begin{equation*}
        \hrho_{k, j} = \frac{1}{\lambda}\Big(y_j - \ST(y_j, \lambda) + \sum\limits_{m = 2}^{k}(y_j - \MCP(y_j, \tfrac{\lambda}{m}, \tfrac{m}{m-1}))\Big)\enspace.
    \end{equation*}
    \begin{itemize}
      \item  If $\tfrac{\lambda}{k} < y_j \leq \tfrac{\lambda}{k-1}$, the firm-threshold definition implies that for all $m < k$ iterations $\hbeta_{m,j} = 0$, hence
    \begin{align*}
        \hrho_{k, j} &= \frac{1}{\lambda}\Big(y_j - \ST(y_j, \lambda) + \sum\limits_{m = 2}^{k}(y_j - m\ST(y_j, \tfrac{\lambda}{m}))\Big)\\
         &= \frac{1}{\lambda}\Big(y_j + (k-2)y_j + (y_j - k(y_j - \tfrac{\lambda}{k}))\Big)
         = 1\enspace.
    \end{align*}
    Therefore, for $\tfrac{\lambda}{k} < y_j \leq \tfrac{\lambda}{k-1}$, we have $\hbeta_{k+1} = \ST(y_j + \lambda, \lambda) = y_j = \MCP(y, \tfrac{\lambda}{k+1}, \tfrac{k+1}{k})$.

    \item If $0 < y_j \leq \tfrac{\lambda}{k}$ it means that for all $m \leq k$ iterations $\hbeta_{m,j} = 0$, and hence
        $\hrho_{k, j} = \frac{1}{\lambda}\Big((y_j - \ST(y_j, \lambda)) + \sum\limits_{m = 2}^{k}(y_j - m\ST(y_j, \tfrac{\lambda}{m}))\Big) = \frac{1}{\lambda}\Big(y_j + (k-1)y_j)\Big) = \frac{k}{\lambda}y_j$.
    Therefore, when $0 < y_j \leq \tfrac{\lambda}{k}$, we have $\hbeta_{k+1} = \ST(y_j + ky_j, \lambda) = (y_j + ky_j - \lambda)_+ = (k+1)(y_j - \tfrac{\lambda}{k+1})_+ = \MCP(y, \tfrac{\lambda}{k+1}, \tfrac{k+1}{k})$.

    \item If $\tfrac{\lambda}{m^*} < y_j \leq \tfrac{\lambda}{m^*-1}$, for some $2 \leq m^* < k$ we know by the induction assumption that for all $m > m^*$ the estimation is given by $\hbeta_{m, j} = y_j$, for all $m < m^*$ the estimation is given by $\hbeta_{m, j} = 0$ and for $m = m^*$ we have $\hbeta_{m^*, j} = m^*(y - \tfrac{\lambda}{m^*})$, hence we can write
        $\hrho_{k, j} = \big(y_j - \ST(y_j, \lambda) + \sum\limits_{m = 2}^{k}(y_j - m\ST(y_j, \tfrac{\lambda}{m}))\big)/\lambda = \big(y_j + \sum\limits_{m = 2}^{m^*}y_j - m^*(y_j - \tfrac{\lambda}{m^*})\big) / \lambda= 1\enspace.$
    \end{itemize}
  Comparing with the firm-threshold definition provide the expected result.
\end{proof}

\end{document}